\documentclass[a4paper,twoside,11pt]{article}

\usepackage{a4wide, fancyhdr, amsmath, amssymb, mathtools, yfonts}
\usepackage{enumitem}
\usepackage{mathrsfs}
\usepackage{graphicx}
\usepackage{tikz}
\usepackage[all]{xy}
\usepackage[utf8]{inputenc}
\usepackage{amsthm}
\usepackage[english]{babel}
\usepackage{chngcntr}
\usepackage{ifthen}
\usepackage{tikz-cd}
\usepackage{authblk}
\usepackage{hyperref}
\numberwithin{equation}{section}

\newcommand{\Z}{\mathbb{Z}}
\newcommand{\Q}{\mathbb{Q}}
\newcommand{\R}{\mathbb{R}}

\newcommand{\CL}{\mathrm{Cl}}

\newcommand\FF{\mathbb{F}}

\newcommand\Gal{\mathrm{Gal}}

\newcommand\ord{\mathrm{ord}}

\DeclareMathOperator{\rk}{rk}

\DeclareFontFamily{U}{wncy}{}
\DeclareFontShape{U}{wncy}{m}{n}{<->wncyr10}{}
\DeclareSymbolFont{mcy}{U}{wncy}{m}{n}
\DeclareMathSymbol{\Sha}{\mathord}{mcy}{"58} 

\newcommand{\calL}{\mathcal{L}}
\newcommand{\Frob}{\mathrm{Frob}}
\newcommand{\Sel}{\mathrm{Sel}}
\newcommand{\nr}{\mathrm{nr}}
\newcommand{\Hom}{\mathrm{Hom}}
\DeclareMathOperator{\tg}{tg}
\DeclareMathOperator{\res}{res}
\newcommand{\calT}{\mathcal{T}}
\DeclareMathOperator{\im}{im}
\DeclareMathOperator{\Epi}{Epi}

\newtheorem{lemma}{Lemma}[section]
\newtheorem{theorem}[lemma]{Theorem}

\newtheorem{prop}[lemma]{Proposition}
\newtheorem{corollary}[lemma]{Corollary}
\newtheorem{mydef}[lemma]{Definition}

\newtheorem{remark}[lemma]{Remark}

\title{\vspace{-\baselineskip}\sffamily\bfseries Statistics of bad parts of class groups}
\author[1]{Peter Koymans\thanks{Mathematisch Instituut, Postbus 80.010, 3508 TA Utrecht, Netherlands, p.h.koymans@uu.nl}}
\author[2]{Yuan Liu\thanks{Department of Mathematics, 1409 W Green St, Urbana, IL 61801, USA, yyyliu@illinois.edu}}
\affil[1]{Utrecht University}
\affil[2]{University of Illinois Urbana-Champaign}
\date{\today}

\begin{document}
\maketitle

\begin{abstract}
Let $p$ be an odd prime. We give a formula for the bad part of $p$-class groups that is valid for $100\%$ of the abelian $p$-extensions when ordered by product of ramified primes.
\end{abstract}

\section{Introduction}	
	For an imaginary quadratic number field $K/\Q$, Gauss gave an explicit description of the $2$-torsion subgroup $\CL(K)[2]$ of the class group of $K$. More precisely, the rank of $\CL(K)[2] \cong \CL(K)/2\CL(K)$ is determined by the number of primes ramified in $K/\Q$. As a consequence of Gauss's genus theory, the average rank of $\CL(K)[2]$ as $K$ varies over all quadratic extensions is infinite, which shows that the naive adaptation of the Cohen--Lenstra conjectures on the distribution of $p$-parts of class groups of quadratic number fields fails when $p=2$. 
	
	Gerth \cite{Gerth} proposed a modification of the Cohen--Lenstra conjectures: instead of $\CL(K)$, it is $2\CL(K)$ whose 2-part distributes randomly according to the Cohen--Lenstra type of probability measure. 
	Similarly to the $p=2$ case, for an odd prime $p$ and a cyclic degree $p$ extension $K/\Q$, Gauss's genus theory implies that the $\Gal(K/\Q)$-invariant part (which is isomorphic to the $\Gal(K/\Q)$-coinvariant part) of $\CL(K)[p^{\infty}]$ is an $\FF_p$-vector space whose rank equals the number of primes ramified in $K/\Q$. Gerth also made conjectures to cover the case where $p$ is odd \cite{Gerth-odd}. Recently, Gerth's conjectures were proven by Smith \cite{Smith} for $p=2$, and for odd $p$ by the first author and Pagano \cite{KP1} (conditional on GRH) and by Smith \cite{Smi22a, Smi22b} (unconditional).

	In \cite{Liu}, the second author extended the Cohen--Lenstra--Gerth type of conjectures and the corresponding genus theory to study the distribution of class groups of abelian extensions. Let $A$ be an abelian group, and $p$ a prime dividing $|A|$. For an $A$-extension $K/\Q$ (i.e.~a pair $(K/\Q, \iota)$, where $K/\Q$ is a Galois extension and $\iota$ is an isomorphism $\Gal(K/\Q) \cong A$), the $p$-part of the class group of $K$, $\CL(K)[p^{\infty}]$, is naturally a module over the ring $\Z_p[A]/(\sum_{\gamma \in A} \gamma)$. Let $e$ be a nontrivial primitive idempotent of $\Q_p[A]$ and define 
	\begin{equation}\label{eq:def-eZp[A]}
		e\Z_p[A]:= \{ ex \mid x \in \Z_p[A]\} \subset \Q_p[A],
	\end{equation}
	which is a quotient ring of $\Z_p[A]$ via the quotient map $\Z_p[A] \to e\Z_p[A]$ defined by sending $x$ to $ex$. The second author \cite{Liu} proved that the ring $e\Z_p[A]$ is a discrete valuation ring and a lattice of the irreducible $\Q_p[A]$-module $e\Q_p[A]$, and studied the distribution of 
	\[
		e\CL(K):= \CL(K)[p^{\infty}] \otimes_{\Z_p[A]} e\Z_p[A]
	\]
	as $K$ varies over totally real $A$-extensions of $\Q$. Define $\mathfrak{m}_e$ to be the maximal ideal of $e\Z_p[A]$, and for any nonzero proper ideal $I \subset e\Z_p[A]$ and a finite $e\Z_p[A]$-module $M$, define the \emph{$I$-rank of $M$} to be
	\[
		\rk_I e M:= \sum_{i=d}^{\infty} n_i, \text{    if $I=\mathfrak{m}_e^d$ and $M \cong \bigoplus_{i=1}^{\infty} \left(e\Z_p[A] / \mathfrak{m}_e^i\right)^{\oplus n_i}$}.
	\]
	Define the ideal $I_e \subseteq e\Z_p[A]$ by
	\begin{equation}\label{eq:def-Ie}
		I_e:= \bigcap_{0 \neq \gamma \in A} \left(1- \gamma, \sum_{j=1}^{\ord(\gamma)} \gamma^j\right) e\Z_p[A].
	\end{equation}
	By \cite[Theorem~1.1 and Conjecture 12.2]{Liu}, for an $A$-extension $K/\Q$: 
	\begin{enumerate}
		\item when $I_e \subseteq I$, the $I$-rank $\rk_I e\CL(K)$ is bounded below, up to a constant, by the number of primes ramified in $K/\Q$ whose inertia subgroup and decomposition subgroup satisfy some explicit conditions depending on $I$ and $A$;
		\item $I_e \cdot e\CL(K)$ is conjectured to be equidistributed according to a Cohen--Lenstra type of probability measure. 
	\end{enumerate}
	Furthermore, a weighted version of the conjecture for the moments of $I_e\cdot e\CL(K)$ in the function field case is proved in \cite[Theorem~1.2]{Liu}. 
		
	In this paper, we give, for $I_e \subseteq I$, a formula to compute $\rk_I e\CL(K)$ for 100\% of the $A$-extensions $K/\Q$ when ordering the extensions by their product of ramified primes. 
	
	\begin{theorem}\label{thm:main}
		Let $p$ be an odd prime, $A$ a finite abelian $p$-group, and $e$ a nontrivial primitive idempotent of $\Q_p[A]$. For a proper ideal $I$ of $e\Z_p[A]$ containing $I_e$ and an $A$-extension $K/\Q$ with a chosen isomorphism $\iota: \Gal(K/\Q) \to A$, define a set of primes that are \emph{special for $K, \iota, e$ and $I$}:
		\[
			\mathcal{R}(K, \iota, e, I):= \left \{ \textup{finite primes $v$ of $\Q$} \,\bigg|\, 
			{\begin{aligned} 
				&\textup{$\iota(D_v)$ acts trivially on $e\Z_p[A]/I$,} \\ 
				&\textup{the exponent of $e\Z_p[A]/I$ divides $|I_v|$} 
			\end{aligned}} \right\},
		\]
		where $D_v$ and $I_v$ are respectively the decomposition subgroup and inertia subgroup of $K/\Q$ at $v$.
		Then for 100\% of the $A$-extensions $K/\Q$, ordered by the product of ramified primes, we have the equality
		\[
			\rk_I e\CL(K) = \# \mathcal{R}(K, \iota, e, I) +C(A, e\Z_p[A]/I),
		\]
		where $C(A, e\Z_p[A]/I)$ is a constant depending only on the abelian group $A$ and the $\Z_p[A]$-module $e\Z_p[A]/I$.
	\end{theorem}
	
	An explicit formula for the constant $C(A, e\Z_p[A]/I)$ is given in Theorem~\ref{thm:main-complete} and Definition~\ref{def:constant-N}, from which one sees that this constant depends only on cohomological information of the $A$-module $e\Z_p[A]/I$.
	In \cite{Liu}, the second author proved that there is an implicit constant $c$ depending on $A$ and $e\Z_p[A]/I$ such that $\rk_I e\CL(K) \geq \# \mathcal{R}(K, \iota, e, I) + c$. We regard it as an interesting open question to determine the smallest choice of $c$ such that equality is attained infinitely often.
	
	When $A$ is a cyclic or elementary abelian $p$-group, we explicitly compute the constant $C(A, e\Z_p[A]/I)$ and obtain the following corollaries of Theorem~\ref{thm:main}. 

\begin{corollary}\label{cor:multicyclic}
	Let $A=(\Z/p\Z)^{\oplus n}$ for $n >1$ and an odd prime $p$. For any nontrivial primitive idempotent $e$, we have $I_e=p \cdot e\Z_p[A]=\mathfrak{m}_e^{p-1} e\Z_p[A]$. Moreover, there exists a subgroup $H \cong (\Z/p\Z)^{\oplus n-1}$ of $A$ such that the $A$-action on $e\Z_p[A]$ factors through the quotient $A/H$. Let $I$ be a proper ideal of $e\Z_p[A]$ such that $I_e \subseteq I$. 
	\begin{enumerate}[label=(\roman*)]
		\item\label{item:multicyclic-1} When $I=\mathfrak{m}_e$, the following equality holds for 100\% of the $A$-extensions $K/\Q$
			\[
				\rk_I e\CL(K)=\dim_{\FF_p} \CL(K)^A= \#\{\textup{finite primes of $\Q$ ramified in $K/\Q$}\}-n.
			\]
		\item\label{item:multicyclic-2} When $I \subset \mathfrak{m}_e$, the following equality holds for 100\% of the $A$-extensions $K/\Q$
			\[
				\rk_I e\CL(K)=\#\left\{\textup{finite primes $v$ of $\Q$}  \,\bigg|\, 
			{\begin{aligned}
				&\textup{$\iota(D_v) \subseteq H$ and} \\ 
				&\textup{$v$ is ramified in $K/\Q$} 
			\end{aligned}} \right\} -1 .
			\]
	\end{enumerate}
\end{corollary}

\begin{corollary}\label{cor:cyclic}
	Let $A=\Z/p^n\Z$ and $e$ a nontrivial primitive idempotent of $\Q_p[A]$. Let $I$ be a proper ideal of $e\Z_p[A]$ containing $I_e$. Then for 100\% of the $A$-extensions $K/\Q$, ordered by the product of ramified primes, the following equality holds
	\[
		\rk_I e\CL(K)=\# \mathcal{R}(K, \iota, e, I) - 1.
	\]
\end{corollary}

\subsection{Related results}
Our results fall in an area that we would like to coin ``statistical genus theory''. A prominent feature of classical genus theory is that it provides explicit formulas for a certain part of the class group, e.g.~Gauss's genus theory provides an exact formula for the $2$-torsion of the class group of imaginary quadratic fields. Since the work of Gauss, there have been many extensions of genus theory explored in the literature. 

Approximately five years ago, it has been observed in the literature that in certain situations this phenomenon persists except that an exact formula is now only valid for 100\% of the extensions. Such a situation typically arises when one is studying class groups (or more generally, Selmer groups) with rather relaxed local conditions. In this situation, the dual Selmer group vanishes for 100\% of the extensions leading to an exact formula via an application of the Greenberg--Wiles formula.

We will now survey some of these recent results. In the class group setting, there have been the works of Fouvry--Koymans--Pagano \cite{FKP} and Koymans--Morgan--Smit \cite{KMS}. These works give a 100\% formula for the $4$-rank of $K(\sqrt{n})$, where $K$ is a fixed quadratic field and $n$ varies over rational (squarefree) integers. In the elliptic curve setting, we mention the works of Morgan--Paterson \cite{MP} and Chan \cite{Chan}, who both give 100\% formulas for Selmer groups in certain families of elliptic curves. There is also the work of Pagano--Sofos \cite{PS}, which explores this phenomenon in a slightly different setting, namely that of conics. It is worth emphasizing that all of these results are (at the very least suspected to be) of a genuine statistical nature, and the relevant formulas are not always true unlike classical genus theory. For example, Morgan--Paterson prove this in the context of their work in \cite[Section 9]{MP}.

In all of the above works, a key ingredient is the character sum method originally developed by Fouvry--Kl\"uners \cite{FK} and Heath-Brown \cite{HB}. This method was extended and used by various authors, most notably in the work of Smith \cite{Smi22a, Smi22b}. It is precisely the character sum method that is utilized to show 100\% vanishing of the dual Selmer group (typically by a first moment computation), and it is precisely this aspect of the approach that is genuinely statistical.

Finally, we would like to mention the works of Alberts \cite{Alberts} and Alberts--O'Dorney \cite{AO, AO2}. These works do not directly fall inside the area of statistical genus theory, but also rely on a philosophically similar application of the Greenberg--Wiles formula (and also predate the aforementioned results).

\subsection{Methods}
The power of our results lies in its generality, as our results hold for general finite abelian $p$-groups (with $p$ odd), which goes significantly beyond the scope of the previous results in statistical genus theory. We will now discuss the key ingredients that allow us to prove our results for general finite abelian $p$-groups.

Our first step, which is executed in Section \ref{sSelmer}, is to compare the relevant class groups with a ``Selmer group'' (by which we mean the subset of a global cohomology group cut out by local conditions). This comparison is already subtle and involves some statistical work, see Lemma \ref{lStatisticalFormula} and Lemma \ref{lFormulaN}. The local conditions of the relevant Selmer group are then computed in Section \ref{sLocal}. Of particular importance is to show that the dual local conditions are often trivial, which is accomplished in Proposition \ref{prop:pushforward<p} and Lemma \ref{lem:notes-1}. It is this condition that will ultimately lead to the 100\% vanishing of the dual Selmer group. 


Our task is then to show 100\% vanishing of the dual Selmer group. This is achieved by a first moment computation combined with the character sum method. Our main analytic result is Theorem \ref{tMainAnalytic} and its proof occupies the majority of Sections \ref{sChar} to \ref{sMain}. We remark that Theorem \ref{tMainAnalytic} aims to encapsulate the general phenomenon that a Selmer group should be small when it has very stringent local conditions, and we hope that it can also be applied to other settings.

Using the character sum method for general abelian groups has only been done once before in the literature, namely in Koymans--Rome \cite{KR} (based on a parametrization due to Koymans--Pagano \cite{KPMalle}). However, \cite{KR} relies on the fact that the discriminant is unfair so that the problem reduces, after some work, to the case of elementary abelian groups. This is not possible in our work as we count by a fair counting function.

We will now discuss various limitations and potential extensions of our results. The restriction that $p$ is odd makes the character sum method significantly easier to apply. A key step in our arguments is, in the terminology of Heath-Brown and Fouvry--Kl\"uners, to classify ``maximal unlinked sets''. However, in the case $p = 2$, the maximal unlinked sets are significantly harder to classify (namely, in that setting, the notion of unlinked must be adapted to account for quadratic reciprocity). It is likely possible that our results can be extended to general global fields with the caveat that $p$ must now not divide the number of roots of unity and the characteristic of the global field.

Finally, another interesting direction is to rigorously show that our results are of a genuine statistical nature, i.e.~do not hold for all abelian $A$-extensions. We expect, but do not show, that this is the case precisely when $I \neq \mathfrak{m}_e$. In other words, we expect our results to be ``statistical'' precisely when classical genus theory does not apply. 

\subsection{Notation}
Throughout the paper, we will use the following notations:
\begin{itemize}
\item We write $\Omega(\Q)$ for the set of places of $\Q$.
\item We fix an algebraic closure $\overline{\Q}$ and algebraic closures $\overline{\Q_v}$ for every place $v \in \Omega(\Q)$. All of our number fields are implicitly taken inside this fixed algebraic closure $\overline{\Q}$.
\item We fix embeddings $\iota_v: \overline{\Q} \rightarrow \overline{\Q_v}$ for every $v \in \Omega(\Q)$. This induces embeddings $\iota_v^\ast: G_{\Q_v} \rightarrow G_\Q$.
\item For a number field $K$ and a nonarchimedean prime $w$ of $K$, we write $\mathcal{T}_w$ for the inertia subgroup of $G_{K_w}$. We write $\sigma_w \in \mathcal{T}_w$ for a topological generator of tame inertia. We also fix a lift $\Frob_w \in G_{K_w}$ of Frobenius (from Section \ref{sChar} onwards, we shall normalize these choices, but this normalization is not required in Sections \ref{sLocal} and \ref{sSelmer}).
\item For an abelian group $A$, we mean by an $A$-extension $K/\Q$ an ordered pair $(K/\Q, \iota)$, where $K/\Q$ is a Galois extension and $\iota$ is a isomorphism $\Gal(K/\Q) \cong A$. The set of $A$-extensions is naturally in bijection with $\mathrm{Epi}(G_\Q, A)$, and we shall frequently (and implicitly) switch between the two viewpoints.
\item For an abelian group $A$, an $A$-extension $K/\Q$ and a place $v \in \Omega(\Q)$, we let $D_v, I_v \subseteq A$ be the usual decomposition group and inertia subgroup (which we view as subsets of $A$ via $\iota$).
\end{itemize}

\subsection*{Acknowledgements}
This project was initiated when the second author visited the first author at the Institute for Theoretical Studies (ETH Z\"urich) in February 2024. Both authors wish to thank the institute for its financial support and excellent working conditions. The first author gratefully acknowledges the support of Dr. Max R\"ossler, the Walter Haefner Foundation and the ETH Z\"urich Foundation. The first author also acknowledges the support of the Dutch Research Council (NWO) through the Veni grant ``New methods in arithmetic statistics''. The second author is partially supported by the NSF grant DMS-2200541.

\section{Local conditions}
\label{sLocal}
The main goal of the next two sections is to relate the class group to a (generalized) Selmer group. It will be important to write down somewhat explicitly the relevant local conditions of this Selmer group. In the process, we compute the sizes of the local conditions and give a sufficient condition for a certain dual Selmer group to vanish. The relevance of this dual Selmer group is that it appears naturally upon applying the Greenberg--Wiles formula; our main theorem will ultimately be a consequence of this formula together with a statistical vanishing result for this dual Selmer group that we prove in Sections \ref{sChar} to \ref{sMain}.

\subsection{Size of local conditions}
We start by computing the sizes of the local conditions. We first give an exact formula that depends on so-called special primes, a notion introduced in \cite{Liu}. After having done so, we show that the local conditions for a certain dual Selmer group are always of size bounded by $p$ (except possibly for primes above $p$).

\subsubsection{An explicit formula for the cardinality of local conditions}
	Let $A$ be a finite abelian $p$-group with $p$ odd. Let $\varphi \in \Epi(G_{\Q}, A)$ and write $K:=\overline{\Q}^{\ker \varphi}$.
	
	\begin{mydef}
		Let $M$ be a finite $\Z_p[A]$-module. For each finite prime $v$ of $\Q$, we define the local condition 
		\[
			\calL_v(M):= \res_{w/v}^{-1}(H^1_{\nr}(G_{K_w}, M)) \subseteq H^1(G_{\Q_v}, M),
		\]
		where $w$ is a prime of $K$ lying above $v$ and $\res_{w/v}$ is the natural restriction map $H^1(G_{\Q_v}, M) \to H^1(G_{K_w}, M)$.
	\end{mydef}
	
	\begin{remark}
	\label{rInfRes}
		Note that $\calL_{v}(M)$ is the kernel of the restriction map
	\[
		\calL_{v}(M)= \ker\left(H^1(G_{\Q_v}, M) \to H^1(\calT_{w}, M)^{G_{\Q_v}}\right) = \inf\left(H^1(G_{\Q_v}/\calT_w, M)\right),
	\] 
	where the last equality follows from the inflation-restriction exact sequence. Since $G_{\Q_v}/\calT_v \cong \hat{\Z}$, we have 
	\[
		G_{\Q_v}/\calT_w \cong \hat{\Z} \times_{D_v/I_v} D_v \cong \hat{\Z} \times I_v.
	\] 
	Since $p$ is odd, we observe that $I_v$ must be a cyclic subgroup of $\Gal(K/\Q)$.
	\end{remark}
	
	 Let $e$ be a nontrivial primitive idempotent of $\Q_p[A]$. Then $e\Z_p[A]$,  as defined in \eqref{eq:def-eZp[A]}, is a complete discrete valuation ring with residue field $\FF_p$ \cite[Proposition~2.7]{Liu}. Let $\mathfrak{m}_e$ denote the maximal ideal of $e\Z_p[A]$, and denote
	\begin{equation}\label{eq:M_j}
		M_j := e\Z_p[A]/\mathfrak{m}_e^j.
	\end{equation}
	Define $r_e$ to be the maximal integer such that there exists some element $\gamma \in A$ for which both $1-\gamma$ and $\ord(\gamma)$ annihilate $M_{r_e}$, i.e.~$r_e$ is such that $\mathfrak{m}_e^{r_e}= I_e$.
	
	 Note that $M_j$ can be viewed as a $G_{\Q}$-module via $\varphi \in \Epi(G_{\Q}, A)$. For each place $v$ of $\Q$, we set
	\begin{equation}\label{eq:L-v-j}
		\calL_{v,j} := \calL_{v}(M_j).
	\end{equation}
		
	\begin{mydef}\label{def:special-prime}
		Given an $A$-extension $K/\Q$, we say that a finite place $v$ of $\Q$ is \emph{special at level $j$ (for $K/\Q$)} if the following conditions are satisfied
		\begin{enumerate}[label=(\alph*)]
			\item the decomposition subgroup $D_v$ of $K/\Q$ at $v$ acts trivially on $M_j$,
			\item the size of the inertia subgroup of $K/\Q$ at $v$ annihilates $M_j$, i.e.~$|I_v| M_j=0$.
		\end{enumerate} 
	\end{mydef}
	
	The main point of this definition is that we can now state our next proposition, which gives an explicit formula for the size of the local conditions.
	
	\begin{prop}\label{prop:local-cond-size}
		Retain the notation and definitions above. For all $1\leq j \leq r_e$ and all primes $v$ of $\Q$ ramified in $K/\Q$, we have 
		\begin{equation}\label{eq:local-cond-size}
			\frac{|\calL_{v,j}|}{|H^0(G_{\Q_v}, M_j)|} = p^{\max\{1\leq i \leq j \, :\, \textup{$v$ is special at level $i$}\}}.
		\end{equation}
	\end{prop}
	
	\begin{proof}
		We will prove the proposition by induction on $j$. When $j=1$, $M_1$ is a simple $\FF_p[A]$-module, so $M_1$ is the trivial module $\FF_p$. By Definition~\ref{def:special-prime}, all primes ramified in $K/\Q$ are special at level 1, so we need to show $|\calL_{v,1}|=p^2$. 
		Since $K/\Q$ is ramified at $v$, we have either $v \equiv 1 \bmod p$ or $v = p$. Since $p$ is odd, we conclude in both cases that $H^1(G_{\Q_v}, \FF_p) \cong \FF_p ^{\oplus 2}$ by local class field theory. The assumption that $K_w/\Q_v$ is ramified implies that the restriction of any $\varphi \in H^1(G_{\Q_v}, \FF_p) = \Hom(G_{\Q_v}, \FF_p)$ to $G_{K_w}$ is unramified, so 
		$$
		\res_{w/v}(H^1(G_{\Q_v}, \FF_p)) \subseteq H^1_{\nr}(G_{K_w}, \FF_p).
		$$ 
		Hence $\calL_{v,1}=H^1(G_{\Q_v}, \FF_p)$ has size equal to $p^2$.
		
		Assume now that \eqref{eq:local-cond-size} holds for some $j < r_e$. We will prove the proposition for $j+1$ in two cases.
		
		{\bf Case 1: suppose $G_{\Q_v}$ acts trivially on $M_{j+1}$.} Then we have 
		$$
		|H^0(G_{\Q_v},M_{j+1})|=p |H^0(G_{\Q_v}, M_j)|,
		$$
		and $v$ is special at level $j+1$ if and only if $|I_v| M_{j+1}=0$. In this case, we have for both modules $M \in \{M_j, M_{j+1}\}$ the equality
		\[
			|\calL_{v}(M)| = |H^1(G_{\Q_v}/\calT_w, M)|=|\Hom(G_{\Q_v}/\calT_w, M)|=|\Hom(\hat{\Z} \times I_v, M)|= |M| \times |M[|I_v|]|
		\]
		by Remark \ref{rInfRes}. If $v$ is special at level $j+1$, then $|I_v| M = 0$, so
		\[
			|\calL_{v, j+1}|=|M_{j+1}|^2 = p^2 |M_j|^2 = p^2 |\calL_{v, j}|.
		\]
		Then it immediately follows that \eqref{eq:local-cond-size} holds.
		If $v$ is not special at level $j+1$, then $M_{j+1}[|I_v|]$ is contained in the maximal submodule $\mathfrak{m}_e M_{j+1}$ of $M_{j+1}$. Since $e\Z_p[A]$ is a discrete valuation ring, $M_j$ is isomorphic to $\mathfrak{m}_eM_{j+1}$, so we have $|M_j[|I_v|]| = |M_{j+1}[|I_v|]|$. Then the arguments above show that $|\calL_{v, j+1}|= p |\calL_{v, j}|$. We conclude that \eqref{eq:local-cond-size} holds also in this case.
		
		{\bf Case 2: suppose $G_{\Q_v}$ acts nontrivially on $M_{j+1}$.} Then we have the equality 
		\begin{equation}
		\label{eInvariants}
		|H^0(G_{\Q_v}, M_{j+1})|=|H^0(G_{\Q_v}, M_{j})|,
		\end{equation}
		because $M_{j+1}^{G_{\Q_v}}$ is contained inside $\mathfrak{m}_e M_{j+1}\cong M_j$. By assumption, $D_v$ acts nontrivially on $M_{j+1}$, so $v$ is not special at level $j+1$. Therefore it suffices to show the equality $|\calL_{v,j+1}|=|\calL_{v, j}|$, which we do now.
		
		For $M \in \{M_j, M_{j+1}\}$, the inflation-restriction exact sequence associated to $1 \to I_v \to G_{\Q_v}/\calT_w \to \hat{\Z} \to 1$ gives the short exact sequence
		\begin{equation}\label{eq:direct-sum-L}
			1 \longrightarrow H^1(\hat{\Z}, M^{I_v}) \longrightarrow H^1(G_{\Q_v}/\calT_w, M) \longrightarrow H^1(I_v, M)^{G_{\Q_v}}\longrightarrow 1.
		\end{equation}
		By Remark \ref{rInfRes}, observe that the middle term in the exact sequence \eqref{eq:direct-sum-L} has cardinality $|\calL_{v, j}|$ for $M = M_j$ and $|\calL_{v, j+1}|$ for $M = M_{j+1}$. We will now proceed by showing that 
		\begin{equation}
		\label{eClaim1}
		|H^1(\hat{\Z}, M_{j+1}^{I_v})| = |H^1(\hat{\Z}, M_j^{I_v})|
		\end{equation}
		and 
		\begin{equation}
		\label{eClaim2}
		|H^1(I_v, M_{j+1})^{G_{\Q_v}}| = |H^1(I_v, M_j)^{G_{\Q_v}}|, 
		\end{equation}
		which clearly suffices.
		
		By \cite[Proposition~1.7.7]{NSW}, $H^1(\hat{\Z}, M^{I_v})$ is isomorphic to the $\hat{\Z}$-coinvariant $(M^{I_v})_{\hat{\Z}}$ of $M^{I_v}$; hence taking cardinalities yields $|H^1(\hat{\Z}, M^{I_v})|=|(M^{I_v})_{\hat{\Z}}|=|(M^{I_v})^{\hat{\Z}}|=|M^{G_{\Q_v}}|$. By combining this with equation \eqref{eInvariants}, we obtain 
		$$
		|H^1(\hat{\Z}, M_{j+1}^{I_v})|=|M_{j+1}^{G_{\Q_v}}|=|M_j^{G_{\Q_v}}|=|H^1(\hat{\Z}, M_{j}^{I_v})|,
		$$
		thus giving \eqref{eClaim1}.
		
		Let $\sigma$ be a generator of $I_v$. Then there is a functorial isomorphism 
		$$
		H^1(I_v, M) \cong \left\{a \in M \mid \left(\sum_{i=1}^{|I_v|} \sigma^i\right)a=0\right\}/(\sigma-1)M
		$$ 		
		by \cite[Proposition~1.7.1]{NSW}. If $\sum_{i=1}^{|I_v|} \sigma^i$ annihilates $M$, then $H^1(I_v, M) \cong M/(\sigma-1)M=M_{I_v}$. Since $e\Z_p[A]$ is a discrete valuation ring, we have $M_{I_v} \cong M^{I_v}$ and hence
		\begin{equation}
		\label{eNormAnn}
		|H^1(I_v, M)^{G_{\Q_v}}|=|(M_{I_v})^{G_{\Q_v}}|=|M^{G_{\Q_v}}|. 
		\end{equation}
		Therefore, if $\sum_{i=1}^{|I_v|} \sigma^i$ annihilates $M_{j + 1}$, and thus also $M_j$, we deduce from equations \eqref{eInvariants} and \eqref{eNormAnn} that
		$$
		|H^1(I_v, M_{j+1})^{G_{\Q_v}}|=|M_{j+1}^{G_{\Q_v}}|=|M_j^{G_{\Q_v}}|=|H^1(I_v, M_{j})^{G_{\Q_v}}|.
		$$ 
		Hence we get equation \eqref{eClaim2} in this case.
		
		For the rest of the proof, we assume that  $\sum_{i=1}^{|I_v|} \sigma^i$ does not annihilate $M_{j+1}$. By \cite[Lemma~2.6]{Liu}, $1-\sigma$ must annihilate $e\Z_p[A]$, and hence annihilates both $M_j$ and $M_{j + 1}$. Then $I_v$ acts trivially on each $M \in \{M_j, M_{j + 1}\}$, and we conclude that
		\[
			H^1(I_v, M)^{G_{\Q_v}}\cong \left\{a \in M \mid \left(\sum_{i=1}^{|I_v|} \sigma^i\right)a=0\right\}^{G_{\Q_v}} = M^{G_{\Q_v}}[|I_v|].
		\]
		It follows from $M_{j+1}^{G_{\Q_v}}\cong M_j^{G_{\Q_v}}$ (see equation \eqref{eInvariants}) that $H^1(I_v, M_{j+1})^{G_{\Q_v}} \cong H^1(I_v, M_j)^{G_{\Q_v}}$. Thus \eqref{eClaim2} is now proven in this case as well.
	\end{proof}
	
\subsubsection{Bounds for the local conditions of the dual Selmer group}	
	Throughout this section, we let $M:= e\Z_p[A]/\mathfrak{m}_e^d$ for some integer $d>0$ and write $M^*:= \Hom(M, \overline{\Q}^\ast)$. Let $f$ be a surjection $M \to \FF_p$ of $e\Z_p[A]$-modules and
		\[
			f^*: H^1(G_{\Q_v}, \mu_p) \longrightarrow H^1(G_{\Q_v}, M^*)
		\]
		the map induced by the dual map $\mu_p \to M^{*}$ of $f$. Write
		\[
			\calL_v:= \calL_v(M)=\ker \left( H^1(G_{\Q_v}, M) \to H^1(\calT_w, M)^{G_{\Q_v}} \right),
		\]
		and define $\calL_v^{\perp}$ to be the orthogonal complement of $\calL_v$ under the local duality pairing
		\[
			H^1(G_{\Q_v}, M) \times H^1(G_{\Q_v}, M^*) \overset{\cup}{\longrightarrow} H^2(G_{\Q_v}, \overline{\Q_v}^\ast) \overset{\mathrm{inv}}{\longrightarrow} \Q/\Z.
		\] 		
	The goal of this subsection is to prove the following proposition.
	\begin{prop}\label{prop:pushforward<p} Retain the notation from above. Then for all finite places $v \neq p$
		\[
			| {f^*}^{-1} (\calL_v^{\perp})| \leq p.
		\]
	\end{prop}
	
	We need several lemmas.
	
	\begin{lemma}\label{lem:notes-1}
		Let $\sigma$ denote a generator of the inertia subgroup $I_v$ of $K/\Q$. If we have $(\sum_{i=1}^{|I_v|} \sigma^i)M=0$, then $\calL_v^{\perp}=0$.
	\end{lemma}
	
	\begin{proof}
		Since $M$ is a $\Z_p[A]$-module and $A$ is an abelian $p$-group, the $\calT_v$-action on $M$ factors through the pro-$p$ completion $\calT_v(p)$, so there is a natural isomorphism
		\[
			H^1(\calT_v, M) \cong H^1(\calT_v(p), M).
		\]
		As $v\neq p$, $\calT_v(p) \cong \Z_p$, thus $H^1(\calT_v, M) \cong M_{\calT_v(p)} = M_{\calT_v}$.
		By \cite[Proposition~1.7.1]{NSW} and the assumption $(\sum_{i=1}^{|I_v|} \sigma^i)M=0$, we have $H^1(I_v, M) \cong M/(\sigma-1)M=M_{I_v}$. Because the $\calT_v$-action on $M$ factors through $I_v$, the above arguments imply that $H^1(\calT_v, M)$ and $H^1(I_v,M)$ have the same size, so the inflation map 
		\[
			H^1(I_v, M) \longrightarrow H^1(\calT_v, M)
		\]
		is an isomorphism. Thus, the restriction map $H^1(\calT_v, M) \to H^1(\calT_w, M)$ is the zero map. Finally, because the restriction map $H^1(G_{\Q_v}, M) \to H^1(\calT_w, M)$ factors through $H^1(\calT_v, M)$, we have $\calL_v = H^1(G_{\Q_v}, M)$, so its orthogonal complement $\calL_v^{\perp}$ is trivial.
	\end{proof}
	
	\begin{lemma}\label{lem:notes-2}
		Let $v \neq p$ be a finite place of $\Q$. If $G_{\Q_v}$ acts trivially on $M^*$ or $\mu_p \not\subset \Q_v$, then $f^*$ is injective; otherwise, $|\ker f^*| =p$.
	\end{lemma}
	\begin{proof}
		The sequence $1 \to \mu_p \to M^* \to M^*/\mu_p \to 1$ induces a long exact sequence 
		\[
			1 \longrightarrow \mu_p^{G_{\Q_v}} \longrightarrow (M^*)^{G_{\Q_v}} \longrightarrow (M^*/\mu_p)^{G_{\Q_v}} \overset{\delta}{\longrightarrow} H^1(G_{\Q_v}, \mu_p) \overset{f^*}{\longrightarrow} H^1(G_{\Q_v}, M^*).
		\]
		If $G_{\Q_v}$ acts trivially on $M^*$, then $(M^*)^{G_{\Q_v}} = M^*$ and $(M^*/\mu_p)^{G_{\Q_v}} = M^*/\mu_p$, so $\delta=0$ and hence $f^*$ is injective. If $G_{\Q_v}$ acts nontrivially on $M^*$, then $(M^*)^{G_{\Q_v}}$ is a proper submodule of $M^*$. Note that $M^*$ has a unique maximal submodule, which is $(M/\FF_p)^*$, so $(M^*)^{G_{\Q_v}} \cong ((M/\FF_p)^*)^{G_{\Q_v}}$. Then because $M/\FF_p \cong \mathfrak{m}_e M$ as $e\Z_p[A]$-modules, we have   
		\begin{equation}
		\label{eSmallerFixed}
			(M^*)^{G_{\Q_v}}  \cong ((M/\FF_p)^*)^{G_{\Q_v}} 
			\cong ((\mathfrak{m}_e M)^*)^{G_{\Q_v}} = (M^*/\mu_p)^{G_{\Q_v}},
		\end{equation}
		where the last equality is because $\mu_{p^{\infty}}$ is divisible. Then from the long exact sequence above, we have $|\ker f^*|=|\im \delta|=|\mu_p^{G_{\Q_v}}|$. This is $p$ if $\mu_p\subset \Q_v$ and is $1$ otherwise.
	\end{proof}
	
	We are now ready to prove Proposition~\ref{prop:pushforward<p}.
	
	\begin{proof}[Proof of Proposition~\ref{prop:pushforward<p}]
		Let $\sigma$ be a generator of the inertia subgroup $I_v$. When we are in the case $(\sum_{i=1}^{|I_v|} \sigma^i)M=0$, Lemma~\ref{lem:notes-1} and Lemma~\ref{lem:notes-2} imply that $|{f^*}^{-1}(\calL_v^{\perp})|\leq p$. For the rest of the proof, we assume  $(\sum_{i=1}^{|I_v|} \sigma^i)M\neq 0$, in which case, by \cite[Lemma~2.6]{Liu}, $\sigma$ must act trivially on $M$. 
		Because $\Q_v(\mu_p^{\infty})/\Q_v$ is unramified, $\calT_v$ acts trivially on the $G_{\Q_v}$-module $\mu_{p^{\infty}}$, so $\sigma$ also acts trivially on $M^*$. 
		
		Consider the commutative diagram
		\begin{equation}\label{eq:diag-1}
		\begin{tikzcd}
			& & H^1_{\nr}(G_{\Q_v}, \mu_p) \arrow[hook]{d} \arrow["f_{\nr}^*"]{r} & H^1_{\nr}(G_{\Q_v}, M^*) \arrow[hook]{d} & \\
			\arrow{r} & \left(\dfrac{M^*}{\mu_p}\right)^{G_{\Q_v}} \arrow{d}\arrow["\delta"]{r} & H^1(G_{\Q_v}, \mu_p) \arrow["f^*"]{r} \arrow["\res_1"]{d}& H^1(G_{\Q_v}, M^*) \arrow{r} \arrow["\res_2"]{d} & { } \\
			\arrow{r} & \left(\dfrac{M^*}{\mu_p}\right)^{\calT_v} \arrow["\delta_{\calT}"]{r} & H^1(\calT_v, \mu_p) \arrow["f^*_{\calT}"]{r} & H^1(\calT_v, M^*) \arrow{r} & { }
		\end{tikzcd}
		\end{equation}
		where the rows are the long exact sequences associated to $1 \to \mu_p \to M^* \to M^*/\mu_p \to 1$ and the columns are inflation-restriction exact sequences. Since $\calT_v$ acts trivially on $M^*$, the connecting map $\delta_{\calT}$ is zero, so $f_{\calT}^*$ is injective. By \cite[Theorem~7.2.15]{NSW}, $H^1_{\nr}(G_{\Q_v}, M^*)$ and $H^1_{\nr}(G_{\Q_v}, M)$ are orthogonal complements in the local duality pairing. Since $H_{\nr}^1(G_{\Q_v}, M) \subseteq \calL_v$ by definition of $\calL_v$, we deduce that $\calL_v^\perp \subseteq H_{\nr}^1(G_{\Q_v}, M^\ast)$. In particular, it follows that 
		\begin{equation}\label{eq:Lperp-kerRes}
			\calL_v^{\perp} \subseteq  \ker \res_2.
		\end{equation}
			
		Suppose $G_{\Q_v}$ acts trivially on $M^*$ or $\mu_p \not\subset \Q_v$. 
		The diagram \eqref{eq:diag-1} and the inclusion \eqref{eq:Lperp-kerRes}, together with the injectivity of $f^*_{\calT}$, imply that ${f^*}^{-1}(\calL_v^{\perp}) \subseteq \ker f^*_{\calT} \circ \res_1 = \ker \res_1= H^1_{\nr}(G_{\Q_v}, \mu_p)$, which clearly has size at most $p$. 
		
		
		Suppose $G_{\Q_v}$ acts nontrivially on $M^*$ and $v \equiv 1 \bmod p$. In this case, we have already seen the isomorphism $(M^*)^{G_{\Q_v}} \cong (M^*/\mu_p)^{G_{\Q_v}}$, see equation \eqref{eSmallerFixed}. Then the first few terms of the long exact sequence used in the top row of \eqref{eq:diag-1}, that is $(\mu_p)^{G_{\Q_v}} \hookrightarrow (M^*)^{G_{\Q_v}} \to (M^*/\mu_p)^{G_{\Q_v}} \overset{\delta}{\to} H^1(G_{\Q_v}, \mu_p)$, imply that $|\im \delta|=p$. As $\delta_{\calT}=0$, we see that $\im \delta \subseteq \ker \res_1$, so we have $\im \delta = \ker \res_1$ by comparing cardinalities. Using this and the injectivity of $f_{\calT}^*$, and chasing the diagram \eqref{eq:diag-1}, we see that $\im f^*$ intersects trivially with $H^1_{\nr}(G_{\Q_v}, M^*) = \ker \res_2$. Then ${f^*}^{-1}(\calL_v^{\perp}) = \im \delta$ has order $p$ by \eqref{eq:Lperp-kerRes}, which completes the proof in this case. 
	\end{proof}

We finish this subsection with an auxiliary result that we will need later. For a prime $v \equiv 1 \bmod p$, we have $\mu_p \subset \Q_v$, so 
$$
H^1(G_{\Q_v}, \mu_p)=\Hom(G_{\Q_v}, \mu_p)= \Hom(\langle \sigma_v \rangle, \mu_p) \times \Hom(\langle \Frob_v\rangle, \mu_p). 
$$
Then there is an isomorphism $H^1(G_{\Q_v}, \mu_p) \to \mu_p\times \mu_p$ that sends $x \in \Hom(G_{\Q_v}, \mu_p)$ to $(x(\sigma_v), x(\Frob_v))$. The next lemma describes the image of $f^{*-1}(\mathcal{L}_v^{\perp})$ under this isomorphism.
	
\begin{lemma}
\label{lLocal}
Let $f: M \rightarrow \mathbb{F}_p$ be a surjection of Galois modules, and denote by $f^\ast: \mu_p \rightarrow M^\ast$ the dual map. Let $v$ be a place not dividing $|M|$ with $v \equiv 1 \bmod p$. 

Then ${f^{\ast}}^{-1}(\mathcal{L}_v^\perp)$, viewed as a subgroup of $H^1(G_{\Q_v}, \mu_p) \cong \mu_p \times \mu_p$ by evaluating at $\sigma_v$ and $\Frob_v$, is determined by $f$, $\varphi(\iota_v^\ast(\sigma_v))$ and $\varphi(\iota_v^\ast(\Frob_v))$.
\end{lemma}

	\begin{proof}
		By Remark \ref{rInfRes}, we have
		\[
			\calL_v =  \text{inf}\left(H^1(G_{\Q_v}/\calT_w, M)\right).
		\]
		Write $d$ for the order of $\varphi(\iota_v^*(\sigma_v))$. Since $\Gal(K_w/\Q_v)$ is abelian and $v\neq p$, $G_{\Q_v}/\calT_w$ is isomorphic to $\hat{\Z} \times \Z/d\Z$. More precisely, the exact sequence
		\[
			1 \longrightarrow \calT_v/\calT_w \longrightarrow G_{\Q_v}/\calT_w \longrightarrow G_{\Q_v}/\calT_v \longrightarrow 1
		\]
		splits via the section $s$ sending the Frobenius element $\Frob \in G_{\Q_v}/\calT_v$ to the image of $\Frob_v$ inside $G_{\Q_v}/\calT_w$. This splitting induces a commutative diagram
		\begin{equation*}
		\begin{tikzcd}
			1 \arrow[r] & \calT_v/\calT_w \arrow[r] & G_{\Q_v}/\calT_w \arrow[r] & G_{\Q_v}/\calT_v \arrow[r] & 1 \\
			1 \arrow[r] & \Z/d\Z \arrow[r] \arrow[u, "g_1"] & \hat{\Z} \times \Z/d\Z \arrow[r] \arrow[u, "g_2"] & \hat{\Z} \arrow[r] \arrow[u, "g_3"] & 1
		\end{tikzcd}
		\end{equation*}
		where $g_1$ sends $1$ to the image of $\sigma_v$ inside $\calT_v/\calT_w$, $g_3$ sends $1$ to $\Frob$ and $g_2$ sends $(1, 0)$ to the image of $\Frob_v$ inside $G_{\Q_v}/\calT_w$ and $(0, 1)$ to the image of $\sigma_v$ inside $G_{\Q_v}/\calT_w$. So we can write $\calL_v$ as (the inflation of) $H^1(\hat{\Z} \times \Z/d\Z, M)$ for which the $\hat{\Z} \times \Z/d\Z$-action on $M$ is determined by $\varphi(\iota_v^*(\Frob_v))$ and $\varphi(\iota_v^*(\sigma_v))$. Indeed, under the identifications above, the action of $(1, 0)$ is precisely the action of $\varphi(\iota_v^*(\Frob_v))$, while the action of $(0, 1)$ is exactly the action of $\varphi(\iota_v^*(\sigma_v))$. 
		
		Then under the map $f: H^1(G_{\Q_v}, M) \to H^1(G_{\Q_v}, \FF_p)$, the image of $\calL_v$ is the image of the map $H^1(\hat{\Z}\times \Z/d\Z, M) \to H^1(\hat{\Z} \times \Z/d\Z, \FF_p)$ induced by $f:M \to \FF_p$, where we are identifying the group $\hat{\Z}\times \Z/d\Z$ with the quotient $G_{\Q_v}/\calT_w$ of $G_{\Q_v}$ as above. Then it follows that $f(\calL_v)$, viewed as a subgroup of $\FF_p \times \FF_p$ by evaluating at $\sigma_v$ and $\Frob_v$, is determined by $f$, $\varphi(\iota_v^{\ast}(\sigma_v))$ and $\varphi(\iota_v^{\ast}(\Frob_v))$. Now consider the following commutative diagram
		\begin{equation*}
		\begin{tikzcd}[column sep = 0.1]
			H^1(G_{\Q_v}, M) \arrow[d, "f"] & \times & H^1(G_{\Q_v}, M^\ast) & \xlongrightarrow{\cup} & H^2(G_{\Q_v}, \mu_{p^\infty}) \arrow[d, "="] \cong \Q_p/\Z_p
\\
			H^1(G_{\Q_v}, \mathbb{F}_p) & \times & H^1(G_{\Q_v}, \mu_p) \arrow[u, "f^\ast"] & \xlongrightarrow{\cup} & H^2(G_{\Q_v}, \mu_{p^\infty}) \cong \Q_p/\Z_p.
		\end{tikzcd}
		\end{equation*}
		For every $x \in H^1(G_{\Q_v}, M)$ and $y \in H^1(G_{\Q_v}, \mu_p)$, we have $x \cup f^\ast(y)=f(x) \cup y$; so we obtain ${f^{\ast}}^{-1}(\calL_v^{\perp})=f(\calL_v)^{\perp}$. Finally, when given $f(\calL_v)$ as a subgroup of $\FF_p \times \FF_p$ as above, one can compute $f(\calL_v)^{\perp}$, viewed as a subgroup of $\mu_p \times \mu_p$ by evaluating at $\sigma_v$ and $\Frob_v$. Then the proof is complete.	
	\end{proof}

\subsection{Vanishing of the dual Selmer group}
	Recall that $\Omega(\Q)$ denotes the set of all (nonarchimedean or archimedean) places of $\Q$.
	
	\begin{mydef}
		Let $M$ be a finite $G_{\Q}$-module. For every $\calL=(\calL_v)_{v\in \Omega(\Q)}$ with $\calL_v$ a subgroup of $H^1(G_{\Q_v}, M)$ such that $\calL_v = H^1_{\textup{nr}}(G_{\Q_v}, M)$ for all but finitely many $v$, we define 
		\[
			\Sel_{\calL}(G_{\Q}, M):=\ker \left( H^1(G_{\Q}, M) \longrightarrow \prod_{v\in \Omega(\Q)} H^1(G_{\Q_v}, M)/ \calL_v\right).
		\]
		
	\end{mydef}
	
	For every finite prime $v$ of $\Q$, recall that $\calL_{v,j}$ denotes the local condition defined in \eqref{eq:L-v-j}. At the archimedean place $v=\infty$, we define $\calL_{\infty, j}:= H^1(G_{\R}, M_j)$, which is 0 because $p$ is odd. Denote by $f_r: M_r \rightarrow \mathbb{F}_p$ the natural surjection of Galois modules.
		
	\begin{lemma}\label{lem:Selmer-vanish}
		Let $M_j$ be the $\Z_p[G_{\Q}]$-module defined in \eqref{eq:M_j}. Let $\calL_j:=(\calL_{v,j})_{v \in \Omega(\Q)}$ and $\calL_j^{\perp}:=(\calL_{v,j}^{\perp})_{v \in \Omega(\Q)}$.  
		If $\Sel_{{f_r^*}^{-1}(\calL_r^{\perp})}(G_{\Q}, \mu_p)=0$ for all $1 \leq r \leq j$, then $\Sel_{\calL_j^{\perp}}(G_{\Q}, M_j^*)=0$.
	\end{lemma}
	
	\begin{proof}
		The lemma obviously holds for $j=1$ because $M_1^*=\mu_p$ and $f_1^*$ is an isomorphism. Suppose the lemma holds for some $j=k$, and we will prove the lemma for $j=k+1$. Recall that $M_{k+1}^*/\mu_p \cong M_k^*$. Associated to the short exact sequence
		\[
			1 \longrightarrow \mu_p \overset{f_{k + 1}^*}{\longrightarrow} M_{k+1}^* \overset{g}{\longrightarrow} M_k^* \longrightarrow 1,
		\]
		we have the commutative diagram (of solid arrows)
		$$
		\begin{tikzcd}
			\Sel_{f_{k + 1}^{*-1}(\calL_{k+1}^{\perp})}(G_{\Q},\mu_p) \arrow[hook]{d} \arrow[dashed]{r} & \Sel_{\calL_{k+1}^{\perp}}(G_{\Q}, M_{k+1}^*) \arrow[hook]{d} \arrow[dashed]{r} & \Sel_{g(\calL_{k+1}^{\perp})}(G_{\Q}, M_k^*) \arrow[hook]{d} \\
			H^1(G_{\Q}, \mu_p) \arrow{d}\arrow{r} &  H^1(G_{\Q}, M_{k+1}^*) \arrow{d}\arrow{r} & H^1(G_{\Q}, M_k^*)  \arrow{d} \\
			\prod\limits_v\dfrac{H^1(G_{\Q_v}, \mu_p)}{f_{k + 1}^{*-1}(\calL_{v,k+1}^{\perp})} \arrow[hook]{r}& \prod\limits_v\dfrac{H^1(G_{\Q_v}, M_{k+1}^*)}{\calL_{v,k+1}^{\perp}} \arrow{r} & \prod\limits_v\dfrac{H^1(G_{\Q_v}, M_k^*)}{g(\calL_{v,k+1}^{\perp})}
		\end{tikzcd}
		$$
		where the products in the last row run over all places of $\Q$ and $g(\calL_{k+1}^{\perp})$ is the image of $\calL_{k+1}^{\perp}$ under the map $H^1(G_{\Q_v}, M_{k+1}^*) \to H^1(G_{\Q_v}, M_k^*)$ induced by $g$. By chasing this diagram, all the dashed maps in the top row exist and are exact at the middle term. Since $\Sel_{f_{k + 1}^{*-1}(\calL_{k+1}^{\perp})}(G_{\Q}, \mu_p)$ is assumed to be zero, we need to prove $\Sel_{g(\calL_{k+1}^{\perp})}(G_{\Q}, M_k^*)=0$. The induction hypothesis says $\Sel_{\calL_k^{\perp}}(G_{\Q_v}, M_k^*)=0$, so it suffices to prove $g(\calL_{v, k+1}^{\perp}) \subseteq \calL_{v, k}^{\perp}$ for every place $v$.
		
		Consider the dual map $g^*: M_k \to M_{k+1}$ of $g$. From the diagram
		\[\begin{tikzcd}
			H^1(G_{\Q_v}, M_k) \arrow["\res"]{r} \arrow["g^*"]{d} & H^1(G_{K_w}, M_k) \arrow["g^*"]{d} \\
			H^1(G_{\Q_v}, M_{k+1}) \arrow["\res"]{r} & H^1(G_{K_w}, M_{k+1})
		\end{tikzcd}\]
		and the fact that $g^*(H^1_{\nr}(G_{K_w}, M_k)) \subseteq H^1_{\nr}(G_{K_w}, M_{k+1})$, we have 
		\begin{equation}\label{eq:g-L}
			g^*(\calL_{v,k}) \subseteq \calL_{v, k+1}.
		\end{equation}
		Finally, consider the diagram
		\begin{equation*}
		\begin{tikzcd}[column sep = 0.1]
			H^1(G_{\Q_v}, M_{k+1}) & \times & H^1(G_{\Q_v}, M_{k+1}^\ast) \arrow[d, "g"]& \xlongrightarrow{\cup} & H^2(G_{\Q_v}, \mu_{p^\infty}) \arrow[d, "="] \cong \Q_p/\Z_p
\\
			H^1(G_{\Q_v}, M_k) \arrow[u, "g^\ast"] & \times & H^1(G_{\Q_v}, M_k^*)  & \xlongrightarrow{\cup} & H^2(G_{\Q_v}, \mu_{p^\infty}) \cong \Q_p/\Z_p.
		\end{tikzcd}
		\end{equation*}
		For all $x \in \calL_{v, k+1}^{\perp}\subseteq H^1(G_{\Q_v}, M_{k + 1}^*)$ and all $y \in \calL_{v,k}$, \eqref{eq:g-L} implies that $g^*(y) \subseteq \calL_{v, k+1}$, so $g^*(y) \cup x =0$. Then from the above diagram one sees that $y \cup g(x) = g^*(y) \cup x =0$. Letting $y$ vary in $\calL_{v, k}$, we deduce that $g(x) \in \calL_{v,k}^{\perp}$, and hence $g(\calL_{v,k+1}^{\perp}) \subseteq \calL_{v, k}^{\perp}$.
	\end{proof}
	
\section{Expressing class groups as Selmer groups}
\label{sSelmer}
In our previous section we have developed the necessary local machinery for our arguments. This section contains the global input from algebraic number theory that will be relevant for us. This consists of two steps with corresponding subsections.

In the first step, we express the relevant class group for us as a Selmer group. We then apply the Greenberg--Wiles formula to compute its size. This formula contains the dual Selmer group, and our future sections are entirely devoted to showing statistical vanishing of this ``dual Selmer group'', thus leading to an algebraic formula valid for 100\% of the $A$-extensions. However, when relating the class group to a Selmer group, a certain second cohomology group shows up in the process.

In the second step, we give a statistical formula for this second cohomology group. Unlike the previous step, it will be fairly straightforward to show that this is a statistical formula, so we do this right away.

Throughout this section, we maintain our convention that $p$ is odd and that $A$ is a finite abelian $p$-group.

\subsection{Algebraic formula for class groups}
We start with a definition.
	
\begin{mydef}
\label{def:R}
Let $D$ be a finite abelian group, $I$ a subgroup of $D$ such that $D/I$ is cyclic, and $M$ a finite $D$-module. Define a profinite group $\widehat{D}$ to be the fiber product
\[
\widehat{D} := \hat{\Z} \times_{D/I} D,
\]
where $\hat{\Z}\to D/I$ and $D\to D/I$ are quotient maps. Then, for a finite $D$-module $M$, we define $R(D,I,M)$ to be the kernel of the inflation map
\[
H^2(D, M) \longrightarrow H^2(\widehat{D}, M).
\]
Note that $R(D,I,M)$ does not depend on the choice of the quotient map $\hat{\Z}\to D/I$.
\end{mydef}

\begin{lemma}
\label{lStatisticalFormula}
Let $p$ be odd and let $A$ be a finite abelian $p$-group. Let $M$ be a finite $A$-module. Let $\varphi\in \mathrm{Epi}(G_{\Q}, A)$ and write $K:= \overline{\Q}^{\ker(\varphi)}$. For each place $v$, we set $\calL_v:= \calL_v(M)$. Then 
	\begin{equation}\label{eq:Cl-Sel}
		\frac{|\mathrm{Sel}_{\calL}(G_{\Q}, M)|}{|H^1(A, M)|}\leq |\Hom_{\mathrm{nr}}(G_K, M)^{A}| \leq \frac{|N| |\mathrm{Sel}_{\calL}(G_{\Q}, M)|}{|H^1(A, M)|}
	\end{equation}
with
\[
N := \ker \left(H^2(A, M) \longrightarrow \bigoplus_{v \in \Omega(\Q)} \frac{H^2(D_v, M)}{R(D_v,I_v, M)}\right),
\]
where $\Hom_{\nr}(G_K,M)^A$ is the set of all $A$-equivariant homomorphisms $G_K \to M$ that are unramified at every place in $\Omega(\Q)$. Furthermore, if we assume that $\mathrm{Sel}_{\mathcal{L}^\perp}(G_\Q, M^\ast) = 0$ with $\calL^{\perp} := (\calL_v^{\perp})_{v \in \Omega(\Q)}$, then the upper bound in \eqref{eq:Cl-Sel} is an equality.
\end{lemma}

\begin{proof}
We have the following commutative diagram 
\begin{equation}
\label{eq:res-diag}
\begin{tikzcd}[column sep = 0.195in]
H^1(A, M) \arrow[hook]{r} \arrow{d}& H^1(G_{\Q}, M) \arrow{d} \arrow["\psi"]{r} & H^1(G_K, M)^{A} \arrow["\tg"]{r}\arrow{d} &H^2 (A, M) \arrow["\inf_1"]{r}\arrow["\res_v^1"]{d} & H^2(G_{\Q}, M) \arrow["\res_v^2"]{d} \\
H^1(D_v, M) \arrow[hook]{r} & H^1(G_{\Q_v}, M) \arrow{r} & H^1(G_{K_w}, M)^{D_v} \arrow["\tg_w"]{r} & H^2(D_v, M) \arrow["\inf_v^1"]{r} & H^2(G_{\Q_v}, M),
\end{tikzcd}
\end{equation}
where the rows are the inflation-restriction short exact sequence and the vertical maps are restriction maps. 
		
When $v$ is nonarchimedean, we have that $\calL_v = \text{inf}\left(H^1(G_{\Q_v}/\calT_w, M)\right)$ by Remark \ref{rInfRes}. It follows by the group structure of $G_{\Q_v}$ that $G_{\Q_v}/\calT_w \cong \widehat{D_v}$, with $\widehat{D_v}$ defined in Definition~\ref{def:R} for $D=D_v$ and $I=I_v$. Restricting everything to Selmer groups, \eqref{eq:res-diag} gives
\begin{equation}
\label{eq:Selmer-diag}
\begin{tikzcd}[column sep = 0.193in]
H^1(A, M) \arrow[hook]{r} \arrow{d} & \mathrm{Sel}_{\calL}(G_{\Q}, M) \arrow{r} \arrow{d} & \Hom_{\nr}(G_K, M)^A \arrow["\tg'"]{r}\arrow{d} & H^2(A, M) \arrow["\res_v^1"]{d} & \\
H^1(D_v, M) \arrow[hook]{r} & \calL_v \arrow{r} & H^1(\hat{\Z}, M)^{D_v} \arrow["\tg'_w"]{r} & H^2(D_v, M) \arrow["\inf_v^2"]{r} & H^2( \widehat{D_v}, M),
\end{tikzcd}
\end{equation}
in which the lower row is obtained by applying the inflation-restriction exact sequence on 
$$
1 \to \ker\left(\hat{\Z} \rightarrow D_v/I_v\right) \to \widehat{D_v} \to D_v \to 1.
$$
We note that $\ker(\hat{\Z} \rightarrow D_v/I_v) \cong \hat{\Z}$, and that the two transgression maps in \eqref{eq:Selmer-diag} are the restrictions of the transgression maps in \eqref{eq:res-diag}. Both rows are exact, and each term in this diagram naturally embeds into the corresponding terms in \eqref{eq:res-diag} by inflation maps. The map $\inf_v^1$ is the composition of $\inf_v^2$ with the inflation map $H^2(\widehat{D_v}, M) \to H^2(G_{\Q_v}, M)$. Then by chasing the diagram \eqref{eq:Selmer-diag}, we see that $\res_v^1(\im \tg') \subseteq \im \tg'_w = \ker \inf_v^2$, and hence 
\begin{equation}
\label{eq:im-nonarch}
\im \tg' \subseteq \ker {\inf}_v^2 \circ \res_v^1 = \ker \left( H^2(A, M) \to \frac{H^2(D_v, M)}{R(D_v, I_v, M)} \right).
\end{equation}
Note that equation \eqref{eq:im-nonarch} is also true when $v$ is archimedean, as $D_\infty$ is trivial for $p$ odd. 

Using \eqref{eq:im-nonarch} for all $v$, we deduce that $\im \tg' \subseteq N$. Hence, by using the top row of the diagram \eqref{eq:Selmer-diag}, we see that 
	\[
		\frac{|\Sel_{\calL}(G_{\Q},M)|}{|H^1(A,M)|} \leq |\Hom_{\nr}(G_K, M)^A| \leq \frac{|N| |\Sel_{\calL}(G_{\Q},M)|}{|H^1(A,M)|},
	\]
	so we obtain \eqref{eq:Cl-Sel}. To prove the last claim in the lemma, it suffices to show that $N = \im \tg'$. Since equation \eqref{eq:im-nonarch} already establishes one inclusion, it remains to show that every class $x \in N$ is contained in $\im \tg'$. 

Recall that $\inf_v^1$ factors through $\inf_v^2$, so the image of $x$ in $H^2(G_{\Q_v}, M)$ is zero for any $v \in \Omega(\Q)$, and hence $\inf_1(x) \in \cap_v \ker\res_v^2$. Note that our assumption $\mathrm{Sel}_{\mathcal{L}^\perp}(G_\Q, M^\ast) = 0$ implies that $\Sha^1(G_\Q, M^\ast) = 0$ and hence $\Sha^2(G_\Q, M) = 0$ by Poitou--Tate duality \cite[Theorem 8.6.7]{NSW}. Therefore the map
\[
H^2(G_{\Q}, M) \to \bigoplus_{v} H^2(G_{\Q_v}, M)
\]
is injective, so $x \in \ker \inf_1 = \im \tg$. 

Let $y \in H^1(G_K, M)^A$ such that $\tg(y)=x$, and let $y_w$ be the image of $y$ in $H^1(G_{K_w}, M)^{D_v}$. Since $x\in N$, we see, by definition of $N$, that $\tg_w (y_w)$ is contained in the kernel of $\inf_v^2: H^2(D_v, M) \to H^2(\widehat{D_v}, M)$ for $v \in \Omega(\Q)$. For each $v\in \Omega(\Q)$, consider the following diagram
\begin{equation}
\label{eLocalLifting}
\begin{tikzcd}
H^1(G_{\Q_v}, M) \arrow{r} \arrow[two heads]{d} & H^1(G_{K_w}, M)^{D_v} \arrow["\tg_w"]{r} \arrow[two heads]{d} & H^2(D_v, M) \arrow{d} \\
H^1(\calT_v, M)^{G_{\Q_v}} \arrow{r} & H^1(\calT_w, M)^{G_{\Q_v}} \arrow{r} & H^2(I_v, M)
\end{tikzcd}
\end{equation}
in which the vertical maps are restriction maps. There is a natural subgroup $S$ of $\widehat{D_v}$ which is mapped isomorphically to $I_v$ under the projection map $\widehat{D_v} \to D_v$. Therefore $\tg_w(y_w) \in \ker \inf_v^2$ implies that the image of $\tg_w(y_w)$ in $H^2(I_v, M)$ is zero. Indeed, this follows from a diagram chase in the diagram
\[
\begin{tikzcd}
H^2(D_v, M) \arrow["\inf_v^2"]{r} \arrow["\res"]{d} & H^2(\widehat{D_v}, M) \arrow["\res"]{d} \\
H^2(I_v, M) \arrow["\cong"]{r} & H^2(S, M)
\end{tikzcd}
\]
induced by the inclusions $I_v \subseteq D_v$ and $S \subseteq \widehat{D_v}$.

Let $z_w$ be the image of $y_w$ in $H^1(\calT_w, M)^{G_{\Q_v}}$ under the middle vertical map in the diagram \eqref{eLocalLifting}. As we just proved that the image of $y_w$ in $H^2(I_v, M)$ is zero, there exists an element $z_v \in H^1(G_{\Q_v}, M)$ whose image is $z_w$. We now claim that 
\[
H^1(G_{\Q}, M) \longrightarrow \bigoplus_{v \in \Omega(\Q)} H^1(G_{\Q_v}, M)/\mathcal{L}_v
\]
is surjective. Indeed, this follows by fixing a place $v$ and then applying the Greenberg--Wiles formula twice: once to $\mathcal{L}$ and once to $\mathcal{L}'$, where $\mathcal{L}'$ has the same local conditions as $\mathcal{L}$ except at $v$ we allow the full $H^1(G_{\Q_v}, M)$. Since the dual Selmer group of $\mathcal{L}$ vanishes by assumption, so does the dual Selmer group of $\mathcal{L}'$. Then the Greenberg--Wiles formula gives
$$
|\mathrm{Sel}_{\mathcal{L}'}(G_\Q, M)| = \frac{|\mathrm{Sel}_{\mathcal{L}}(G_\Q, M)| |H^1(G_{\Q_v}, M)|}{|\mathcal{L}_v|}
$$
from which the claim readily follows.

By the claim, there exists $z \in H^1(G_{\Q}, M)$ mapping to $z_v + \mathcal{L}_v$ for every $v$. Recall that $\psi$ is defined in the diagram \eqref{eq:res-diag}. Then $y - \psi(z)$ is unramified and is mapped to $x$ under $\tg$, so we proved $x \in \im \tg'$.
\end{proof}

\subsection{\texorpdfstring{Statistical formula for $N$}{Statistical formula for N}}
The goal of this subsection is to derive a statistical formula for $N$. We start with a definition.

\begin{mydef}\label{def:constant-N}
Let $A$ be a finite abelian $p$-group and let $M$ be a finite $A$-module. Define $\mathcal{C}$ to be the set of pairs $(D, I)$, where $D$ is a $2$-generated subgroup of $A$ (i.e.~$D$ is a quotient of $\Z^2$) and $I$ is a cyclic subgroup of $D$ with $D/I$ also cyclic. We define
$$
N_{\textup{typical}}(A, M) := \ker\left(H^2(A, M) \rightarrow \bigoplus_{(D, I) \in \mathcal{C}} \frac{H^2(D, M)}{R(D, I, M)}\right).
$$
Given $\varphi \in \mathrm{Epi}(G_\Q, A)$, we also define
$$
N_\varphi(A, M) := \ker\left(H^2(A, M) \longrightarrow \bigoplus_{v \in \Omega(\Q)} \frac{H^2(D_v, M)}{R(D_v, I_v,  M)}\right).
$$
\end{mydef}

Given a homomorphism $\varphi \in \mathrm{Epi}(G_\Q, A)$, we define $\mathfrak{f}(\varphi)$ to be the product of the ramified primes in the fixed field of $\ker(\varphi)$.

\begin{lemma}
\label{lFormulaN}
Let $p$, $A$ and $M$ be as above. Then we have
$$
\lim_{X \rightarrow \infty} \frac{|\{\varphi \in \mathrm{Epi}(G_\Q, A) : \mathfrak{f}(\varphi) \leq X, N_{\textup{typical}}(A, M) \neq N_\varphi(A, M)\}|}{|\{\varphi \in \mathrm{Epi}(G_\Q, A) : \mathfrak{f}(\varphi) \leq X\}|} = 0.
$$
\end{lemma}

\begin{proof}
By the structure of the tame local Galois group, we have $(D_v, I_v) \in \mathcal{C}$ for all $v \neq p$. Moreover, using that $p$ is odd, we also have that $(D_p, I_p) \in \mathcal{C}$. Therefore we always have an inclusion $N_{\text{typical}}(A, M) \subseteq N_\varphi(A, M)$. Since $\mathcal{C}$ is a finite set, it suffices to establish that
$$
\lim_{X \rightarrow \infty} \frac{|\{\varphi \in \mathrm{Epi}(G_\Q, A) : \mathfrak{f}(\varphi) \leq X, \exists v \in \Omega(\Q) \text{ with } D_v = D, I_v = I\}|}{|\{\varphi \in \mathrm{Epi}(G_\Q, A) : \mathfrak{f}(\varphi) \leq X\}|} = 1
$$
for every $(D, I) \in \mathcal{C}$. This is a consequence of Wood's result \cite{Wood} combined with the argument in \cite[Theorem A.1]{Liu}. Alternatively, we prove a much stronger result later in this paper, see Theorem \ref{twabBig}.
\end{proof}

\subsection{\texorpdfstring{Computing $N_{\text{typical}}(A, M)$ in special cases}{Computing Ntypical(A, M) in special cases}}
In general, the authors are unaware of a simple formula for $N_{\text{typical}}(A, M)$. In order to give explicit formulas for the leading constant in Corollaries \ref{cor:multicyclic} and \ref{cor:cyclic}, we will compute $N_{\text{typical}}(A, M)$ in several special cases of interest.

	\begin{lemma}\label{lem:N-in-inf}
		Let $A$ be a finite abelian group. Let $H$ be a finite abelian group and let $C$ be a cyclic group such that $A \cong H \times C$. Let $M$ be an $\Z_p[A]$-module such that $H$ acts trivially on $M$. Assume that $M^C \cong \FF_p$. Then $N_{\textup{typical}}(A, M)$ is contained in the image of the injection $H^2(H, M^C) \to H^2(A, M)$.
	\end{lemma}
	
	\begin{proof}
		By the Hochschild--Serre spectral sequence \cite[Theorem~2.4.6]{NSW}, there is a decomposition
		\begin{equation}\label{eq:H-S}
			H^2(H \times C, M) \cong H^2(H, M^C) \oplus H^1(H, H^1(C, M)) \oplus H^2(C, M).
		\end{equation}
		Moreover, via this isomorphism, the first factor $H^2(H, M^C)$ as a subgroup of $H^2(H \times C, M)$ is the image of the edge homomorphism $E_2^{2,0} \to E^2$ which is the inflation map $H^2(H, M^C) \to H^2(A, M)$. Since $C$ is cyclic, $|H^1(C, M)| = |\widehat{H}^0(C,M)|=|M^C/\textup{Nm}_C M|$, where $\hat{H}^0$ is the Tate cohomology. Because $M^C$ is $\FF_p$, $H^1(C, M)$ is either $\FF_p$ or trivial. 
		Let $\alpha \in H^1(H, H^1(C, M))$. If $M^C$ is trivial then $\alpha$ is zero; otherwise, since $H^1(H, \FF_p)=\Hom(H, \FF_p)$, $\alpha$ is a homomorphism $H \to \FF_p$, and we can decompose $H$ as $H_1 \times C_1$ such that $C_1$ is cyclic and $H_1 \subseteq \ker \alpha$ (this decomposition is not unique). 
		
		Consider the pair $(D,I)$ such that $D=C_1 \times C$ and $I=C$. Then $\widehat{D}= \hat{\Z} \times C$.
		The decomposition \eqref{eq:H-S} is functorial in $H$ for a fixed $C$-module $M$ \cite{Jannsen}, so it commutes with the restriction map $H^2(A,M) \to H^2(D, M)$ and the inflation map $H^2(D, M) \to H^2(\widehat{D}, M)$ for the groups $D$ and $\widehat{D}$ defined above. Then comparing \eqref{eq:H-S} with the decomposition of $H^2(\widehat{D}, M)$
		\[
			H^2(\widehat{D}, M) \cong H^2(\hat{\Z}, M^C) \oplus H^1(\hat{\Z}, H^1(C, M)) \oplus H^2(C, M)
		\]
		and using that $H^2(\hat{\Z}, M^C)=0$, we see that the kernel of $H^2(A, M) \overset{\res}{\longrightarrow} H^2(D, M) \overset{\inf}{\longrightarrow} H^2( \widehat{D}, M)$ is contained in $H^2(H, M^C) \oplus \ker (H^1(H, H^1(C, M)) \to H^1(\hat{\Z}, H^1(C, M)) )$. By our construction of $D$ and $I$, we see that the image of $\alpha$ in 
		$$
		H^1(\hat{\Z}, H^1(C, M)) = H^1(\hat{\Z}, \FF_p)
		$$ 
		is the homomorphism $\hat{\Z} \to C_1 \overset{\alpha|_{C_1}}{\longrightarrow} \FF_p$, which is nonzero if $\alpha \neq 0$.
		
		The above argument shows that any nonzero element of $H^1(H, H^1(C, M)) \times H^2(C,M)$ is not contained in the kernel of $H^2(A, M) \to H^2(\widehat{D},M)$ for some choice of $(D, I)$. Then the lemma follows by definition of $N_{\textup{typical}}(A, M)$.
	\end{proof}

	\begin{lemma}\label{lem:N-Fp}
		For an abelian $p$-group $H$, we have $N_{\textup{typical}}(H, \FF_p)=0$. 
	\end{lemma}
	
	\begin{proof}
		We will show that for any nonzero element $\alpha \in H^2(H, \FF_p)$ there exists a pair $(D,I)$ of subgroups of $H$ such that both $I$ and $D/I$ are cyclic and $\alpha$ is not in the kernel of $H^2(H, \FF_p) \overset{\res}{\longrightarrow} H^2(D, \FF_p) \overset{\inf}{\longrightarrow} H^2 (\widehat{D}, \FF_p)$. Let $1 \to \FF_p \to \widetilde{H} \overset{\pi_{\alpha}}{\longrightarrow} H \to 1$ be the extension of $H$ corresponding to $\alpha$.
		
		Suppose $\widetilde{H}$ is abelian. Then since $\alpha\neq 0$, the group extension $\pi_{\alpha}:\widetilde{H} \to H$ does not split. So there exists a nontrivial cyclic subgroup $D$ of $H$ such that $\pi_{\alpha}^{-1}(D)$ is a cyclic group $\widetilde{D}$. We take $I=D$, then $\widehat{D}=D \times \hat{\Z}$. After taking $\res: H^2(H, \FF_p) \to H^2(D, \FF_p)$ and $\inf: H^2(D, \FF_p) \to H^2(\widehat{D}, \FF_p)$, $\alpha$ is mapped to the class that corresponds to the group extension $1 \to \FF_p \to \widetilde{D} \times \hat{\Z} \to \widehat{D} \to 1$, which does not split because $1\to \FF_p \to \widetilde{D} \to D\to 1$ does not split. So $\alpha$ is not in $\ker( \inf\circ \res)$, and hence $\alpha \not \in N_{\textup{typical}}(H, \FF_p)$.
		
		Suppose $\widetilde{H}$ is nonabelian. Then there exists two elements $\tilde{x}, \tilde{y} \in \widetilde{H}$ such that $[\tilde{x}, \tilde{y}]$ generates $\ker \pi_{\alpha}$. Let $x:=\pi_{\alpha}(\tilde{x})$ and $y:=\pi_{\alpha}(\tilde{y})$, and let $D:= \langle x, y\rangle$ and $I:=\langle x \rangle$. Then $\widetilde{D}:= \pi_{\alpha}^{-1}(D)$ is a nonabelian $p$-group whose abelianization is $D$.
		One can check that $\inf \circ \res(\alpha)$ corresponds to the group extension $1 \to \FF_p \to \widetilde{D} \times_{D/I} \hat{\Z} \to \widehat{D} \to 1$, and this group extension is nonsplit since $\widehat{D}$ is the abelianization of $\widetilde{D} \times_{D/I} \hat{\Z}$. So in this case, $\alpha$ is also not in $\ker( \inf\circ \res)$, and hence $\alpha \not \in N_{\textup{typical}}(H, \FF_p)$.
	\end{proof}
	
	\begin{lemma}\label{lem:N-M-exp=p}
		Let $A$ be an elementary abelian $p$-group, and $e$ a nontrivial primitive idempotent of $\Q_p[A]$. Then $N_{\textup{typical}}(A, e\Z_p[A]/I)=0$ for any proper ideal $I$ of $e\Z_p[A]$ containing $I_e$.
	\end{lemma}
	
	\begin{proof}
		Write $M:=e\Z_p[A]/I$. By \cite[Lemma~2.5]{Liu}, the action of $A$ on $e\Z_p[A]$ factors through a cyclic quotient of $A$. So we can decompose $A$ as $A \cong H \times C$ such that $C \cong \Z/p\Z$ and $H$ acts trivially on $e\Z_p[A]$. 
		Since $A$ is an elementary abelian $p$-group, it follows from the definition of $I_e$ (see \eqref{eq:def-Ie}) that $p(e\Z_p[A])\subseteq I_e$, so $M$ has exponent $p$. Also, $M^{A} \cong \FF_p$ as $e$ is not the trivial idempotent. Then by Lemma~\ref{lem:N-in-inf}, $N_{\textup{typical}}(A,M)$ is contained in the image of the inflation map $H^2(H, \FF_p) \to H^2(A, M)$.
		
		For a pair $(D, I)$ of subgroups of $H$ such that both $I$ and $D/I$ are cyclic, consider the following commutative diagram 
		\[\begin{tikzcd}
			H^2(H, \FF_p) \arrow["\inf", hook]{r} \arrow["\res"]{d} & H^2(A, M) \arrow["\res"]{d}  \arrow["\res"]{dr}& \\
			H^2(D, \FF_p) \arrow["\inf", hook]{r} \arrow["\inf"]{d} & H^2(D \times C, M) \arrow["\res"]{r} \arrow["\inf"]{d} & H^2(D, M) \arrow["\inf"]{d} \\
			H^2(\widehat{D}, \FF_p) \arrow["\inf", hook]{r} & H^2(\widehat{D} \times C, M) \arrow["\res"]{r} & H^2(\widehat{D}, M).
		\end{tikzcd}\]
		The horizontal inflation maps are injective by the decomposition \eqref{eq:H-S}, and the top-left and bottom-right squares commute by \cite[Proposition~1.5.5]{NSW}. One can check on the cochain level that the composition of $\inf$ and $\res$ in the bottom row is the map $H^2(\widehat{D}, \FF_p) \to H^2(\widehat{D}, M)$ induced by the injection $\FF_p \hookrightarrow M$. Since $M$ has exponent $p$ and $\widehat{D}$ acts trivially on $M$, the map $H^1(\widehat{D}, M) \to H^1(\widehat{D}, M/\FF_p)$, which is $\Hom(\widehat{D}, M) \to \Hom(\widehat{D}, M/\FF_p)$, is surjective. So by the long exact sequence associated to $0 \to \FF_p \to M \to M/\FF_p \to 0$, we have that $H^2(\widehat{D}, \FF_p) \to H^2(\widehat{D}, M)$ is injective. Then through any path in the above diagram, the kernel of $H^2(H, \FF_p) \to H^2(\widehat{D}, M)$ equals the kernel of $H^2(H, \FF_p) \to H^2(\widehat{D}, \FF_p)$. 
		
		Finally, considering all pairs $(D,I)$ of subgroups of $H$ such that both $I$ and $D/I$ are cyclic, the lemma immediately follows from Lemma~\ref{lem:N-in-inf} and Lemma~\ref{lem:N-Fp}.
	\end{proof}

\section{Character sums}
\label{sChar}
\subsection{Abstract setup}
It will be convenient to normalize our choices of tame inertia and Frobenius as follows. Let $p$ be an odd prime and fix a primitive $p$-th root of unity $\zeta \in \overline{\Q}$. Let $q \equiv 1 \bmod p$. 

We demand that our choice of $\Frob_q$ projects trivially to $\Gal(\Q_q(\zeta_q)/\Q_q)$. It will be useful later to rephrase this property, so we do this now. We remark that $\Q_q(\zeta_q) = \Q_q(\sqrt[q - 1]{-q})$. Note that both fields are abelian extensions of $\Q_q$ of the same degree. Moreover, $q$ is a norm from both fields, and hence the desired equality follows from local class field theory. Hence $\Frob_q$ projecting trivially to $\Gal(\Q_q(\zeta_q)/\Q_q)$ is equivalent to 
$$
\frac{\Frob_q(\sqrt[q - 1]{-q})}{\sqrt[q - 1]{-q}} = 1.
$$
In this case, we get the commutative diagram
\begin{equation}
\label{eChiq}
\begin{tikzcd}
\Q_q^\ast/\Q_q^{\ast p} \arrow[r] \arrow[d] & \Z_q^\ast/\Z_q^{\ast p} \arrow[d] \\
H^1(G_{\Q_q}, \langle \iota_q(\zeta) \rangle) \arrow[r] & \langle \zeta \rangle
\end{tikzcd}
\end{equation}
where the right vertical map sends $u$ to the local Hilbert symbol $\iota_q^{-1}((q, u)_p)$ and where the bottom horizontal map is evaluation at $\Frob_q$ followed by $\iota_q^{-1}$. Moreover, the top horizontal map comes from choosing $-q$ as a uniformizer, so it sends an element $(-q)^\alpha u^\beta \in \Q_q^\ast/\Q_q^{\ast p}$ with $\alpha, \beta \in \{0, \dots, p - 1\}$ and $v_q(u) = 0$ to $u^\beta$. Using our normalization of $\Frob_q$ and the definition of the power residue symbol in terms of the local Artin symbol, one readily verifies that the diagram commutes. 

Writing $f$ for the map $\Z_q^\ast/\Z_q^{\ast p} \rightarrow \langle \zeta \rangle$, we get an induced map on $\{a \in \Z : \gcd(a, q) = 1\}$ by the formula $f(\iota_q(a))$. For now, let us record the observation that this is a Dirichlet character that we shall henceforth call $\chi_q$. In particular, we may view $\chi_q$ as a map from $G_\Q$ to $\mathbb{C}^\ast$.

As for the generator of tame inertia $\sigma_q$, we normalize it so that
$$
\chi_q(\iota_q^\ast(\sigma_q)) = \zeta.
$$
With these choices, there is some $s(q) \in (\Z/p\Z)^\ast$ such that
\begin{equation}
\label{eInertiaChoice}
\begin{tikzcd}
\Q_q^\ast/\Q_q^{\ast p} \arrow[r, "s(q) \cdot v_q"] \arrow[d] & \Z/p\Z \arrow[d] \\
H^1(G_{\Q_q}, \langle \iota_q(\zeta) \rangle) \arrow[r] & \langle \zeta \rangle
\end{tikzcd}
\end{equation}
where the left vertical map is the Kummer isomorphism, where the right vertical map sends $1$ to $\zeta$ and where the bottom horizontal map is evaluation at $\sigma_q$ (followed by $\iota_q^{-1}$).

Let $A$ be a finite, abelian $p$-group (with $p$ odd). Let $\mathcal{F} = \mathrm{Epi}(G_\Q, A)$. Recall that, given a homomorphism $\varphi \in \mathrm{Epi}(G_\Q, A)$, we defined $\mathfrak{f}(\varphi)$ to be the product of the ramified primes in the fixed field of $\ker(\varphi)$. We denote by $\mathcal{F}(X)$ the subset of $\varphi \in \mathcal{F}$ satisfying $\mathfrak{f}(\varphi) \leq X$.

A tuple $\mathcal{L} = (\mathcal{L}_v)_{v \in \Omega_\Q}$ is called \emph{a tuple of local Selmer conditions} if $\mathcal{L}_v$ is a subgroup of $H^1(G_{\Q_v}, \mu_p)$ such that $\mathcal{L}_v = H^1_{\text{nr}}(G_\Q, \mu_p)$ for all but finitely many $v$. We define $U$ to be the set of all tuples of local Selmer conditions. For a prime $v \equiv 1 \bmod p$, we let $\mathrm{Ev}_v: H^1(G_{\Q_v}, \mu_p) \rightarrow \mu_p \times \mu_p$ be the map given by evaluation at $\sigma_v$ and $\Frob_v$.

\begin{mydef}
\label{dLocal}
Let $C \subseteq A$ be a cyclic subgroup. An assignment of local conditions, special with respect to $C$, is a map $g: \mathcal{F} \rightarrow U$ with the following properties
\begin{itemize}
\item the tuple $g(\varphi) = (\mathcal{L}_v)_{v \in \Omega_\Q}$ satisfies $\mathcal{L}_v = H^1_{\textup{nr}}(G_{\Q_v}, \mu_p)$ for all finite places $v$ not dividing $p \mathfrak{f}(\varphi)$;
\item for $v \equiv 1 \bmod p$, the subgroup $\mathrm{Ev}_v(\mathcal{L}_v)$ depends only on $\varphi(\iota_v^\ast(\sigma_v))$ and $\varphi(\iota_v^\ast(\Frob_v))$;
\item we have $|\mathcal{L}_v| \in \{1, p\}$ for all finite places $v \neq p$;
\item we have $|\mathcal{L}_v| = 1$ for all finite places $v \neq p$ with $\varphi(G_{\Q_v}) = \varphi(\mathcal{T}_v) = C$;
\item we have $\mathcal{L}_p = H^1(G_{\Q_p}, \mu_p)$.
\end{itemize}
\end{mydef}

Our main analytic theorem is the following result. We note that the fifth (and last) condition in Definition \ref{dLocal} is for convenience only. In fact, we can formally deduce from our next theorem that the same result still holds without this extra condition.

\begin{theorem}
\label{tMainAnalytic}
Let $g: \mathcal{F} \rightarrow U$ be an assignment of local conditions, special with respect to a nontrivial subgroup $C \subseteq A$. Then we have
$$
\lim_{X \rightarrow \infty} \frac{\# \{\varphi \in \mathcal{F}(X) : \mathrm{Sel}(g(\varphi)) = 0\}}{\# \mathcal{F}(X)} = 1.
$$
\end{theorem}

Note that it is reasonable to expect Theorem \ref{tMainAnalytic} to hold. Indeed, for a typical $A$-extension, there should be a positive proportion of primes satisfying $\varphi(G_{\Q_v}) = \varphi(\mathcal{T}_v) = C$. Once this happens, we have more local conditions than candidate Selmer elements (this is also known as the Tamagawa ratio in the literature), and thus we should expect Theorem \ref{tMainAnalytic} to hold true. Making this heuristic rigorous will be the main task of the coming sections.

\subsection{Parametrization of abelian extensions}
In this subsection we study a convenient parametrization of abelian extensions, which is reminiscent of the work of Koymans--Rome \cite[Theorem 2.2]{KR}.

\begin{mydef}
Let $p$ be an odd prime and let $A$ be a finite abelian $p$-group. Define $\mathcal{A}$ to be the set of tuples $(w_a)_{a \in A - \{\textup{id}\}}$ with the properties
\begin{itemize}
\item the integers $w_a$ are positive, squarefree and pairwise coprime,
\item we have $q \mid w_a$ implies $q = p$ or $q \equiv 1 \bmod \ord(a)$.
\end{itemize}
Define $\mathrm{Ev}: \mathrm{Hom}(G_\Q, A) \rightarrow \mathcal{A}$ to be the map that sends $\varphi \in  \mathrm{Hom}(G_\Q, A)$ to the unique tuple $(w_a)_{a \in A - \{\textup{id}\}}$ satisfying the property
$$
q \mid w_a \Longleftrightarrow \varphi(\iota_q^\ast(\sigma_q)) = a
$$
for all $a \in A - \{\textup{id}\}$ and all primes $q$.
\end{mydef}

Fix an injection $\iota: \Q/\Z \xhookrightarrow{} \overline{\Q}^\ast$ sending $1/p$ to $\zeta$. Throughout the paper, we shall view $\Z/k\Z$ as a subset of $\Q/\Z$ by sending $1$ to $1/k$. 

Recall that $\chi_q: \Z \rightarrow \mathbb{C}$ is the unique Dirichlet character with period $q$ and of order $p$ that comes from our fixed embedding $\overline{\Q} \rightarrow \overline{\Q_q}$ as explained after the diagram \eqref{eChiq}. Let $q$ be a prime and let $n \geq 1$ be an integer. If $q \equiv 1 \bmod p^n$, then we define $\chi_{q, n}: G_\Q \rightarrow \mathbb{C}^\ast$ to be the unique homomorphism of order $p^n$ that is unramified away from $q$ and satisfies 
$$
\chi_{q, n}(\iota_q^\ast(\sigma_q)) = \iota(1/p^n).
$$
In particular, we observe that $\chi_{q, 1} = \chi_q$ by our conventions on $\sigma_q$ and $\iota(1/p) = \zeta$.

If $q \equiv 1 \bmod p$ but $q \not \equiv 1 \bmod p^n$, then we define $\chi_{q, n}: G_\Q \rightarrow \mathbb{C}^\ast$ as follows: let $m$ be the largest integer such that $q \equiv 1 \bmod p^m$ and take $\chi_{q, n}: G_\Q \rightarrow \mathbb{C}^\ast$ be any map of order $p^m$ such that 
$$
\chi_{q, n}^{p^{n - m}} = \chi_{q, m}.
$$
Moreover, we also define for every integer $n \geq 1$ the homomorphism $\chi_{p, n}: G_\Q \rightarrow \mathbb{C}^\ast$ that is unramified away from $p$ and satisfies
$$
\chi_{p, n}(\iota_p^\ast(\sigma_p)) = \iota(1/p^n).
$$
It is proven in Koymans--Rome \cite[Theorem 2.2]{KR} that $\mathrm{Ev}$ has an inverse, called $\mathrm{Par}$, so in particular $\mathrm{Ev}$ is a bijection between $\mathrm{Hom}(G_\Q, A)$ and $\mathcal{A}$. The map $\mathrm{Par}$ is constructed explicitly in \cite{KR}, but we shall only need the following (characterizing) property of $\text{Par}$.

Note that the subset $\{x \in \overline{\Q}^\ast : x^k = 1\}$ has a natural exponentiation map $\Z/k\Z \times \{x \in \overline{\Q}^\ast : x^k = 1\} \rightarrow \{x \in \overline{\Q}^\ast : x^k = 1\}$, which sends $(a, x)$ to $x^b$, where $b$ is any integer mapping to $a$ in $\Z/k\Z$. We shall simply write $x^a$ for this exponentiation map without further mention.

\begin{theorem}
\label{tMainPar}
Let $A$ be a finite abelian $p$-group. Then for every $\mathbf{w} \in \mathcal{A}$ and every $\psi \in \mathrm{Hom}(A, \Z/p^n\Z)$, the map $\iota(\psi(\mathrm{Par}(\mathbf{w})))$ is a Dirichlet character given by the formula
$$
\iota(\psi(\mathrm{Par}(\mathbf{w}))) = \prod_{a \in A - \{\textup{id}\}} \prod_{q \mid w_a} \chi_{q, n}^{\psi(a)}.
$$
Moreover, we have the formula
$$
\mathfrak{f}(\mathrm{Par}(\mathbf{w})) = \prod_{a \in A - \{\textup{id}\}} w_a,
$$
and the equivalence
$$
\mathrm{Par}(\mathbf{w}) \in \mathrm{Epi}(G_\Q, A) \Longleftrightarrow \langle a \in A - \{\textup{id}\} : w_a \neq 1 \rangle = A.
$$
\end{theorem}

\subsection{Expansion into Dirichlet characters}
Instead of directly detecting triviality of the Selmer group in Theorem \ref{tMainAnalytic}, we shall compute a suitable average, after removing a small set of elements in $\mathcal{F}(X)$. We will justify that the removed set is small only much later in Theorem \ref{twabBig}, the main goal of this section is to combinatorially rewrite the average size of the Selmer group in such a way that our oscillation results in Section \ref{sAnalytic} will apply. By the parametrization Theorem \ref{tMainPar}, we have
$$
\sum_{\varphi \in \mathcal{F}(X)} \# \mathrm{Sel}(g(\varphi)) = 
\sum_{\substack{\mathbf{u} = (u_a)_{a \in A - \{\text{id}\}} \\ \prod_{a \in A - \{\text{id}\}} u_a \leq X \\ q \mid u_a \Rightarrow q = p \text{ or } q \equiv 1 \bmod \ord(a) \\ \langle a \in A - \{\text{id}\} : u_a \neq 1 \rangle = A}} \mu^2\left(\prod_{a \in A - \{\text{id\}}} u_a\right) \cdot \# \mathrm{Sel}(g(\text{Par}(\mathbf{u}))).
$$
Define $G(X) := \exp((\log X)^{1/2})$. Instead of considering the full set $\mathcal{F}(X)$, we define $\mathcal{F}^\sharp(X)$ to be the following subset
$$
\mathcal{F}^\sharp(X) := \left\{\varphi \in \mathcal{F}(X) \, | \, \forall a \in A - \{\text{id}\} \ \forall b \in A : \prod_{\substack{q \text{ prime}, \ q \neq p \\ \varphi(\iota_q^\ast(\sigma_q)) = a, \varphi(\iota_q^\ast(\Frob_q)) = b}} q > G(X) \right\}.
$$
We will now describe the subset of $\mathcal{A}$ given by the image of $\mathcal{F}^\sharp(X)$ under $\mathrm{Ev}$. Let 
$$
\mathcal{J} := \{(a, b) : a \in A - \{\text{id}\}, b \in A\}. 
$$
Henceforth, we shall frequently abuse notation and simply write $\text{Par}(\mathbf{u})(\Frob_q)$ and $\text{Par}(\mathbf{u})(\sigma_q)$ instead of the correct $\text{Par}(\mathbf{u})(\iota_q^\ast(\Frob_q))$ and $\text{Par}(\mathbf{u})(\iota_q^\ast(\sigma_q))$. Given $\mathbf{u}$, we define new variables $\mathbf{w} = (w_{a, b})_{(a, b) \in \mathcal{J}}$ and $\mathbf{f} = (f_a)_{a \in A - \{\text{id}\}}$ through
$$
w_{a, b} = \prod_{\substack{q \mid u_a, \ q \neq p \\ \text{Par}(\mathbf{u})(\Frob_q) = b}} q, \quad \quad f_a = \begin{cases}
p &\text{if } p \mid u_a \\
1 &\text{if } p \nmid u_a.
\end{cases}
$$
Given $\mathbf{f}$ and $\mathbf{w}$, we can recover the old variables $\mathbf{f w} := (u_a)_{a \in A - \{\text{id}\}}$ via the formula
$$
u_a = f_a \prod_{b \in A} w_{a, b}.
$$
With this notation set up, now observe that $\mathrm{Par}(\mathbf{f w}) \in \mathcal{F}^\sharp(X)$ precisely when $w_{a, b} > G(X)$ and $\prod_{a \in A - \{\text{id}\}} f_a \prod_{a \in A - \{\text{id}\}, b \in A} w_{a, b} \leq X$. Note that $\mathrm{Par}(\mathbf{f w})$ is indeed an epimorphism for sufficiently large $X$; this follows from $w_{a, b} > G(X) \geq 1$ and Theorem \ref{tMainPar}. Hence we have
\begin{multline*}
\sum_{\varphi \in \mathcal{F}^\sharp(X)} \# \mathrm{Sel}(g(\varphi)) = \sum_{\substack{\mathbf{w} = (w_{a, b})_{(a, b) \in \mathcal{J}}, \mathbf{f} = (f_a)_a \\ \prod_a f_a \prod_{(a, b) \in \mathcal{J}} w_{a, b} \leq X, \ w_{a, b} > G(X) \\ q \mid w_{a, b} \Rightarrow q \equiv 1 \bmod \ord(a) \\ q \mid f_a \Rightarrow q = p}} \mu^2\left(\prod_a f_a \prod_{(a, b) \in \mathcal{J}} w_{a, b}\right) \cdot \\
\# \mathrm{Sel}(g(\text{Par}(\mathbf{f w}))) \cdot \mathbf{1}_{q \mid w_{a, b} \Rightarrow \text{Par}(\mathbf{f w})(\Frob_q) = b},
\end{multline*}
We continue to expand $\# \mathrm{Sel}(g(\text{Par}(\mathbf{f w})))$. Suppose that $q \equiv 1 \bmod p$ divides $w_{a, b}$. The local condition $\mathcal{L}_q$ depends only on the values of inertia $\mathrm{Par}(\mathbf{f w})(\iota_q^\ast(\sigma_q)) = a$ and Frobenius $\mathrm{Par}(\mathbf{f w})(\iota_q^\ast(\Frob_q)) = b$, i.e.~$\mathcal{L}_q$ is determined by the pair 
$$
(a, b) \in \mathcal{J}
$$ 
thanks to the second bullet point of Definition \ref{dLocal}. Thus it makes sense to define $\mathcal{J}'$ to be the subset of $\mathcal{J}$ where $\mathcal{L}_q$ is contained inside the unramified local conditions, and $\mathcal{J}''$ to be the subset of $\mathcal{J}$ where $\mathcal{L}_q = 0$. By the ramification constraints on $\mathrm{Sel}(g(\text{Par}(\mathbf{f w})))$, we see that any Selmer element $\kappa \in H^1(G_\Q, \langle \zeta \rangle) \cong \Q^\ast/\Q^{\ast p}$ is of the shape
\begin{align}
\label{eKappa}
\kappa = p^\alpha \prod_{i = 0}^{p - 1} d_i^i
\end{align}
with $d_0, \dots, d_{p - 1}$ coprime integers multiplying to $\prod_{j \not \in \mathcal{J}'} w_j$ and $\alpha \in \{0, \dots, p - 1\}$. Note that $\iota_v(\kappa)$ satisfies the local Selmer conditions from Definition \ref{dLocal} at all places $v$ except possibly at primes $q$ dividing $w_j$ for some $j \not \in \mathcal{J}'$ or $j \in \mathcal{J}''$. We will now check the local conditions at those primes $q$.

Recall that $\chi_q: \Z \rightarrow \mathbb{C}$ is the unique Dirichlet character with period $q$ and of order $p$ that comes from our fixed embedding $\overline{\Q} \rightarrow \overline{\Q_q}$ as explained after the diagram \eqref{eChiq}. With this notation, we can detect the local condition at $q$ as follows. First suppose that $q \mid w_j$ with $j \in \mathcal{J}''$. By equation \eqref{eKappa}, we see that $q \nmid \kappa$, and hence the local condition is equivalent to $\kappa$ being a $p$-th power modulo $q$. We detect this condition with the expression
\begin{align}
\label{eDetect1}
\mathbf{1}_{\res_q(\kappa) = 0} = \frac{1}{p} \sum_{c \in \mathbb{F}_p} \chi_q^c(\kappa).
\end{align}
Next suppose that $q \mid w_j$ with $j \not \in \mathcal{J}'$, so $q \mid d_\nu$ for some unique index $\nu \in \{0, \dots, p - 1\}$. Write $\delta: \Q_q^\ast/\Q_q^{\ast p} \rightarrow H^1(G_{\Q_q}, \langle \iota_q(\zeta) \rangle)$ for the local Kummer map at $q$. In this case, it follows from the second bullet point of Definition \ref{dLocal} that the local condition is of the shape
\begin{align}
\label{eejLocal}
\delta(\iota_q(\kappa))\left(\Frob_q \sigma_q^{e(j)}\right) = 1
\end{align}
for some exponent $e(j)$ depending only on $j \not \in \mathcal{J}'$. Since $\delta(\iota_q(\kappa))(\Frob_q) = \chi_q(\kappa/(-q)^\nu)$ by equation \eqref{eChiq} and $\delta(\iota_q(\kappa))(\sigma_q^{e(j)}) = \iota_q(\zeta)^{\nu e(j) s(q)}$ by equation \eqref{eInertiaChoice}, we may detect the condition \eqref{eejLocal} with the expression
\begin{align}
\label{eDetect2}
\mathbf{1}_{\delta(\iota_q(\kappa))\left(\Frob_q \sigma_q^{e(j)}\right) = 1} = \frac{1}{p} \sum_{c \in \mathbb{F}_p} \zeta^{c \nu e(j) s(q)} \chi_q^c(\kappa/(-q)^\nu).
\end{align}
By combining \eqref{eDetect1} and \eqref{eDetect2}, it follows that $\# \mathrm{Sel}(g(\text{Par}(\mathbf{f w})))$ equals
$$
\sum_{\alpha \in \mathbb{F}_p} \sum_{d_0 \cdots d_{p - 1} = \prod_{j \not \in \mathcal{J}'} w_j} \prod_{j \not \in \mathcal{J}'} \prod_{q \mid w_j} \left(\frac{1}{p} \sum_{c \in \mathbb{F}_p} \zeta^{c \nu e(j) s(q)} \chi_q^c(\kappa/(-q)^\nu)\right) \prod_{j \in \mathcal{J}''} \prod_{q \mid w_j} \left(\frac{1}{p} \sum_{c \in \mathbb{F}_p} \chi_q^c(\kappa)\right),
$$
where $\nu$ is the unique index $\nu \in \{0, \dots, p - 1\}$ with $q \mid d_\nu$ and where $\kappa$ is given by equation \eqref{eKappa}. Before we continue, it will be useful to make another change of variables. For each $(a, b) \not \in \mathcal{J}'$ and each $\nu \in \{0, \dots, p - 1\} \cong \mathbb{F}_p$, we introduce the variable
$$
w_{a, b, \nu} = \gcd(w_{a, b}, d_\nu).
$$
We define a corresponding set of indices by
\begin{align*}
\mathcal{I}_1 = &\{(a, b) : a \in A - \{\text{id}\}, b \in A, (a, b) \in \mathcal{J}'\} \cup \\
&\{(a, b, \nu) : a \in A - \{\text{id}\}, b \in A, (a, b) \not \in \mathcal{J}', \nu \in \mathbb{F}_p\},
\end{align*}
and we write $\mathcal{I}'_1 := \mathcal{I}_1 - \mathcal{J}'$ (this equals the second set in the above formula). Therefore the average of $\# \mathrm{Sel}(g(\text{Par}(\mathbf{f w})))$ over all $\mathbf{f}$ and $\mathbf{w}$ equals
\begin{multline*}
\sum_{\substack{\mathbf{w} = (w_i)_{i \in \mathcal{I}_1} \\ \mathbf{f} = (f_a)}}^\flat \sum_{\alpha \in \mathbb{F}_p} \prod_{(a, b, \nu) \in \mathcal{I}'_1} \prod_{q \mid w_{a, b, \nu}} \left(\frac{1}{p} \sum_{c \in \mathbb{F}_p} \zeta^{c \nu e(a, b) s(q)} \chi_q^c\left(p^\alpha (-q)^{-\nu} \prod_{(a', b', \nu') \in \mathcal{I}'_1} w_{a', b', \nu'}^{\nu'}\right)\right) \times \\
\prod_{(a, b) \in \mathcal{J}''} \prod_{q \mid w_{a, b}} \left(\frac{1}{p} \sum_{c \in \mathbb{F}_p} \chi_q^c\left(p^\alpha \prod_{\substack{(a', b', \nu') \in \mathcal{I}'_1}} w_{a', b', \nu'}^{\nu'}\right)\right),
\end{multline*}
where the summation conditions $\flat$ on the outer sum over $w_i$, with $i = (a, b)$ or $i = (a, b, \nu)$, and $f_a$ are:
\begin{gather*}
q \mid w_i \Rightarrow q \equiv 1 \bmod \ord(a), \quad \mu^2\left(\prod_{i \in \mathcal{I}_1} w_i\right) = 1, \quad f_a \mid p, \quad  \mu^2\left(\prod_a f_a\right) = 1 \\
\prod_a f_a \prod_{i \in \mathcal{I}_1} w_i \leq X, \quad q \mid w_i \Longleftrightarrow \mathrm{Par}(\mathbf{f} \mathbf{w})(\Frob_q) = b \\
(a, b) \in \mathcal{J}' \Longrightarrow w_{a, b} > G(X), \quad (a, b) \not \in \mathcal{J}' \Longrightarrow \prod_{\nu \in \mathbb{F}_p} w_{a, b, \nu} > G(X).
\end{gather*}
Here we apply the parametrization map to $\mathbf{f w}$ by changing back to the old variables in the natural way. 

To remove the term $(-q)^{-\nu}$ in our Dirichlet characters, we now introduce \emph{modified Dirichlet characters}. These will go against the standard conventions for Dirichlet characters. Our definition is based on \cite[Definition 5.3]{Smi22a}.

\begin{mydef}
Recall that we have chosen a Dirichlet character $\chi_q$ (of period $q$ and order $p$) for each prime $q \equiv 1 \bmod p$. We define a modified Dirichlet character $\psi_q: \Z_{\geq 1} \rightarrow \mathbb{C}^\ast$ to be the unique strongly multiplicative function that is given on primes $p$ by
$$
\psi_q(p) = 
\begin{cases}
\chi_q(p) &\textup{if } p \neq q \\
1 &\textup{if } p = q.
\end{cases}
$$
Moreover, if $n$ is a squarefree integer only divisible by primes $1$ modulo $p$, we set
$$
\psi_n = \prod_{q \mid n} \psi_q.
$$
\end{mydef}

With this definition set, we rewrite the average of $\# \mathrm{Sel}(g(\text{Par}(\mathbf{f w})))$ over all $\mathbf{f}$ and $\mathbf{w}$ as
\begin{multline*}
\sum_{\substack{\mathbf{w} = (w_i)_{i \in \mathcal{I}_1} \\ \mathbf{f} = (f_a)}}^\flat \sum_{\alpha \in \mathbb{F}_p} \prod_{(a, b, \nu) \in \mathcal{I}'_1} \prod_{q \mid w_{a, b, \nu}} \left(\frac{1}{p} \sum_{c \in \mathbb{F}_p} \zeta^{c \nu e(a, b) s(q)} \psi_q^c\left(p^\alpha (-1)^\nu \prod_{(a', b', \nu') \in \mathcal{I}'_1} w_{a', b', \nu'}^{\nu'}\right)\right) \times \\
\prod_{\substack{(a, b) \in \mathcal{J}''}} \prod_{q \mid w_{a, b}} \left(\frac{1}{p} \sum_{c \in \mathbb{F}_p} \psi_q^c\left(p^\alpha \prod_{(a', b', \nu') \in \mathcal{I}'_1} w_{a', b', \nu'}^{\nu'}\right)\right).
\end{multline*}
Define the index sets
\begin{align*}
\mathcal{I}_2' := &\{(a, b, c, \nu) : a \in A - \{\text{id}\}, b \in A, \nu, c \in \mathbb{F}_p : (a, b) \not \in \mathcal{J}'\} \\
\mathcal{I}_2'' := &\{(a, b, c) : a \in A - \{\text{id}\}, b \in A, c \in \mathbb{F}_p : (a, b) \in \mathcal{J}''\} \\
\mathcal{I}_2''' := &\{(a, b): a \in A - \{\text{id}\}, b \in A, (a, b) \in \mathcal{J}' - \mathcal{J}''\} \\
\mathcal{I}_2 :=& \mathcal{I}_2' \cup \mathcal{I}_2'' \cup \mathcal{I}_2'''.
\end{align*}
Write $a(i)$, $b(i)$, $c(i)$ and $\nu(i)$ for respectively the $a$-coordinate, $b$-coordinate, $c$-coordinate and $\nu$-coordinate of an element $i \in \mathcal{I}_2$, and write $h(\mathbf{w}, \alpha)$ for the multiplicative function in $\mathbf{w}$ given by
$$
h(\mathbf{w}, \alpha) := \prod_{i \in \mathcal{I}_2' \cup \mathcal{I}_2''} \frac{\psi_{w_i}^{c(i)}(p^\alpha)}{p^{\omega(w_i)}} \prod_{i \in \mathcal{I}_2'} \prod_{q \mid w_i} \zeta^{c(i) \nu(i) e(a(i), b(i)) s(q)}.
$$
For odd order Dirichlet characters, we always have $\psi_q(-1) = 1$. Using this and expanding the above two products shows that the total sum equals
\begin{align}
\label{eSum}
\sum_{\substack{\mathbf{w} = (w_i)_{i \in \mathcal{I}_2} \\ \mathbf{f} = (f_a)}}^{\flat \flat} \sum_{\alpha \in \mathbb{F}_p} h(\mathbf{w}, \alpha) \prod_{\substack{k \in \mathcal{I}_2' \\ \ell \in \mathcal{I}_2'}} \psi_{w_k}^{c(k)}\left(w_\ell^{\nu(\ell)}\right) \prod_{\substack{k \in \mathcal{I}_2'' \\ \ell \in \mathcal{I}_2'}} \psi_{w_k}^{c(k)}\left(w_\ell^{\nu(\ell)}\right).
\end{align}
The summation conditions $\flat \flat$ in the sum \eqref{eSum} restrict the summation over $w_i$ to positive squarefree, pairwise coprime integers satisfying the conditions
\begin{gather*}
q \mid w_i \Rightarrow q \equiv 1 \bmod \ord(a(i)), \quad f_a \mid p, \quad \prod_a f_a \prod_{i \in \mathcal{I}_2} w_i \leq X \\
q \mid w_i \Longleftrightarrow \mathrm{Par}(\mathbf{f} \mathbf{w})(\Frob_q) = b(i), \quad \mu^2\left(\prod_a f_a\right) = 1 \\
\forall a \in A - \{\text{id}\} \ \forall b \in A : \prod_{\substack{i \in \mathcal{I}_2 \\ a(i) = a, b(i) = b}} w_i > G(X).
\end{gather*}
We now detect the condition $q \mid w_i$ implies $\text{Par}(\mathbf{f w})(\Frob_q) = b(i)$. Define $A^\vee = \Hom(A, \Q/\Z)$ and recall that we have fixed an injection $\iota: \Q/\Z \xhookrightarrow{} \mathbb{C}^\ast$. For $q \mid w_i$, we have the equality 
$$
\mathbf{1}_{\text{Par}(\mathbf{f w})(\Frob_q) = b} = \frac{1}{|A|} \sum_{\chi \in A^\vee} \iota\left(\chi(\text{Par}(\mathbf{f w})(\Frob_q) - b)\right).
$$
Define the index sets
\begin{align}
\label{eI3}
\mathcal{I}_3' := &\{(a, b, \chi, c, \nu) : a \in A - \{\text{id}\}, b \in A, \chi \in A^\vee, \nu, c \in \mathbb{F}_p : (a, b) \not \in \mathcal{J}'\} \nonumber \\
\mathcal{I}_3'' := &\{(a, b, \chi, c) : a \in A - \{\text{id}\}, b \in A, \chi \in A^\vee, c \in \mathbb{F}_p : (a, b) \in \mathcal{J}''\} \nonumber \\
\mathcal{I}_3''' := &\{(a, b, \chi): a \in A - \{\text{id}\}, b \in A, \chi \in A^\vee, (a, b) \in \mathcal{J}' - \mathcal{J}''\} \nonumber \\
\mathcal{I}_3 :=& \mathcal{I}_3' \cup \mathcal{I}_3'' \cup \mathcal{I}_3'''.
\end{align}
Denote by $\chi(k)$ the $\chi$-coordinate of some index $k \in \mathcal{I}_3$ and denote by $\tilde{h}(\mathbf{w}, \alpha)$ the multiplicative function given by
\begin{equation}
\label{eDefTildeh}
\tilde{h}(\mathbf{w}, \alpha) := \prod_{i \in \mathcal{I}_3} \left(\frac{\iota(\chi(i)(-b(i)))}{|A|}\right)^{\omega(w_i)} \prod_{i \in \mathcal{I}_3' \cup \mathcal{I}_3''} \frac{\psi_{w_i}^{c(i)}(p^\alpha)}{p^{\omega(w_i)}} \prod_{i \in \mathcal{I}_3'} \prod_{q \mid w_i} \zeta^{c(i) \nu(i) e(a(i), b(i)) s(q)}.
\end{equation}
Expansion and changing variables one more time yields 
$$
\sum_{\substack{\mathbf{w} = (w_i)_{i \in \mathcal{I}_3} \\ \mathbf{f} = (f_a)}}^{\flat \flat \flat} \sum_{\alpha \in \mathbb{F}_p} \tilde{h}(\mathbf{w}, \alpha) \prod_{\substack{k \in \mathcal{I}_3' \\ \ell \in \mathcal{I}_3'}} \psi_{w_k}^{c(k)}\left(w_\ell^{\nu(\ell)}\right) \prod_{\substack{k \in \mathcal{I}_3'' \\ \ell \in \mathcal{I}_3'}} \psi_{w_k}^{c(k)}\left(w_\ell^{\nu(\ell)}\right) \prod_{\ell \in \mathcal{I}_3} \iota\left(\chi(\ell)(\mathrm{Par}(\mathbf{f w})(\Frob_{w_\ell}))\right),
$$
where $\flat \flat \flat$ restricts the summation to positive squarefree, pairwise coprime integers satisfying
\begin{gather}
q \mid w_i \Rightarrow q \equiv 1 \bmod \ord(a(i)), \quad f_a \mid p, \quad \prod_a f_a \prod_{i \in \mathcal{I}_3} w_i \leq X \nonumber \\ 
\mu^2\left(\prod_a f_a\right) = 1, \quad \forall a \in A - \{\text{id}\} \ \forall b \in A : \prod_{\substack{i \in \mathcal{I}_3 \\ a(i) = a, b(i) = b}} \hspace{-0.2cm} w_i > G(X). \label{eSummationConditions}
\end{gather}
Finally, we use Theorem \ref{tMainPar} to rewrite $\iota \circ \chi \circ \text{Par}(\mathbf{f w})$. We end this section by formally stating what we have proven so far. 

\begin{theorem}
\label{tCharSum}
We have the identity
\begin{multline*}
\sum_{\varphi \in \mathcal{F}^\sharp(X)} \# \mathrm{Sel}(g(\varphi)) = \sum_{\substack{\mathbf{w} = (w_i)_{i \in \mathcal{I}_3} \\ \mathbf{f} = (f_a)}}^{\flat \flat \flat} \sum_{\alpha \in \mathbb{F}_p} \tilde{h}(\mathbf{w}, \alpha) \times \\
\prod_{\substack{k \in \mathcal{I}_3' \\ \ell \in \mathcal{I}_3'}} \psi_{w_k}^{c(k)}\left(w_\ell^{\nu(\ell)}\right) \prod_{\substack{k \in \mathcal{I}_3'' \\ \ell \in \mathcal{I}_3'}} \psi_{w_k}^{c(k)}\left(w_\ell^{\nu(\ell)}\right) \prod_{\ell \in \mathcal{I}_3} \prod_{k \in \mathcal{I}_3} \prod_{q \mid w_k} \chi_{q, \log_p \ord(\chi(\ell))}(w_\ell)^{\chi(\ell)(a(k))} \times \\
\prod_{\ell \in \mathcal{I}_3} \prod_{a \in A - \{\textup{id}\}} \chi_{f_a, \log_p \ord(\chi(\ell))}(w_\ell)^{\chi(\ell)(a)},
\end{multline*}
where the summation conditions are in \eqref{eSummationConditions} and the function $\tilde{h}$ is defined in equation \eqref{eDefTildeh}.
\end{theorem}

\section{Combinatorics}
The main result of this section is Theorem \ref{tUnlinkedClassification}. Before we embark on the proof, we start with two lemmas.

\begin{lemma}
\label{lExtendingF}
Let $A$ be a finite abelian group and let $S \subseteq A$. Let $f_0: S \rightarrow \mathbb{F}_p$ be a function. Assume that for all $a_1, \dots, a_m \in S$ and $k_1, \dots, k_m \in \Z$ with $k_1 a_1 + \dots + k_m a_m = 0$ we have
$$
\sum_{i = 1}^m k_i f_0(a_i) = 0.
$$
Then there exists a homomorphism $f: \langle S \rangle \rightarrow \mathbb{F}_p$ such that $f(s) = f_0(s)$ for all $s \in S$.
\end{lemma}

\begin{proof}
Denote by $\Z^S$ the free abelian group on the set $S$. Using the universal property of free abelian groups, we get induced homomorphisms $\overline{f_0}$ and $\overline{\iota}$ (coming from respectively the map $f_0: S \rightarrow \FF_p$ and the inclusion $\iota: S \hookrightarrow A$) that we depict in the following diagram
\begin{equation*}
\begin{tikzcd}
\Z^S \arrow[r, "\overline{f_0}"] \arrow[dr, swap, two heads, "\overline{\iota}"] & \mathbb{F}_p \\
& \langle S \rangle. 
\end{tikzcd}
\end{equation*}
Now the assumption in the lemma is precisely equivalent to $\ker(\overline{\iota}) \subseteq \ker(\overline{f_0})$. Therefore we get an induced map 
$$
\frac{\Z^S}{\ker(\overline{\iota})} \xlongrightarrow{\overline{f_0}} \mathbb{F}_p.
$$
Exploiting the isomorphism $\Z^S/\ker(\overline{\iota}) \cong \langle S \rangle$ gives the desired homomorphism.
\end{proof}

\begin{lemma}
\label{lGlue}
Let $A$ be a finite abelian group and let $S, T \subseteq A$. Let $f: \langle S \rangle \rightarrow \mathbb{F}_p$ and $g: \langle T \rangle \rightarrow \mathbb{F}_p$ be homomorphisms. Assume that for all $a_1, \dots, a_m \in S$, all $k_1, \dots, k_m \in \Z$, all $b_1, \dots, b_n \in T$ and all $l_1, \dots, l_n \in \Z$ with $k_1 a_1 + \dots + k_m a_m = l_1 b_1 + \dots + l_n b_n$ we have
$$
\sum_{i = 1}^m k_i f(a_i) = \sum_{j = 1}^n l_j g(b_j).
$$
Then there exists a homomorphism $h: \langle S \cup T \rangle \rightarrow \mathbb{F}_p$ such that $h(s) = f(s)$ for all $s \in S$ and $h(t) = g(t)$ for all $t \in T$.
\end{lemma}

\begin{proof}
It is easy to see that $f = g$ on $S \cap T$. We now consider the two induced homomorphisms on the free group $\varphi: \Z^{S \cup T} \rightarrow \mathbb{F}_p$ and $\psi: \Z^{S \cup T} \twoheadrightarrow \langle S \cup T \rangle$. Our assumption precisely reads that $\ker(\psi) \subseteq \ker(\varphi)$. The resulting map $\widetilde{\varphi}: \Z^{S \cup T}/\ker(\psi) \rightarrow \mathbb{F}_p$ is the desired homomorphism.
\end{proof}

We recall the definitions of $\mathcal{I}_3$, $\mathcal{I}_3'$, $\mathcal{I}_3''$ and $\mathcal{I}_3'''$ from \eqref{eI3}.

\begin{mydef}
Call two elements $u_1, u_2 \in \mathcal{I}_3$ unlinked if
$$
\chi(u_1)(a(u_2)) + \mathbf{1}_{u_1, u_2 \in \mathcal{I}_3'} \cdot c(u_2) \nu(u_1) + \mathbf{1}_{u_1 \in \mathcal{I}_3', u_2 \in \mathcal{I}_3''} \cdot c(u_2) \nu(u_1) = 0.
$$
We say that a subset $U$ of $\mathcal{I}_3$ is unlinked if
\begin{itemize}
\item every two distinct $u_1, u_2 \in U$ are unlinked, and
\item the projection map $\pi: U \rightarrow (A - \{\textup{id}\}) \times A$ is surjective.
\end{itemize}
We say that $U$ is maximal unlinked if $U$ is unlinked and moreover every strict superset is not unlinked. In other words, $U$ is maximal under inclusion among the unlinked subsets of $\mathcal{I}_3$.
\end{mydef}

The main goal of this section is to (almost) classify the maximal unlinked sets. This is achieved in our next theorem, for which we set up some notation now. 

\begin{mydef}
Write $p_1: \mathcal{I}_3 \rightarrow A - \{\textup{id}\}$ for the projection map on the first coordinate. We say that a homomorphism $f: A \rightarrow \mathbb{F}_p$ is admissible if it vanishes on $p_1(\mathcal{I}_3''')$.
\end{mydef}

Define
\begin{align*}
\mathcal{U} := &\{(a, b, 0): a \in A - \{\textup{id}\}, b \in A, (a, b) \in \mathcal{J}' - \mathcal{J}''\} \cup \\
&\{(a, b, 0, c) : a \in A - \{\textup{id}\}, b \in A, c \in \mathbb{F}_p, (a, b) \in \mathcal{J}''\} \cup \\
&\{(a, b, 0, c, 0) : a \in A - \{\textup{id}\}, b \in A, c \in \mathbb{F}_p, (a, b) \not \in \mathcal{J}'\}
\end{align*} 
and define for each admissible homomorphism $f$
\begin{align*}
\mathcal{U}_f := &\{(a, b, 0): a \in A - \{\textup{id}\}, b \in A, (a, b) \in \mathcal{J}' - \mathcal{J}''\} \cup \\
&\{(a, b, 0, f(a)) : a \in A - \{\textup{id}\}, b \in A, (a, b) \in \mathcal{J}''\} \cup \\
&\{(a, b, x \cdot f, f(a), -x) : a \in A - \{\textup{id}\}, b \in A, x \in \mathbb{F}_p, (a, b) \not \in \mathcal{J}'\}.
\end{align*} 

\begin{theorem}
\label{tUnlinkedClassification}
Let $d := \dim_{\FF_p} A/pA \geq 1$. Let $U \subseteq \mathcal{I}_3$ be a maximal unlinked subset. Then we have that 
$$
U \in \{\mathcal{U}\} \cup \{\mathcal{U}_f : f \textup{ admissible}\}
$$
or there exists an element $u_0 \in U$ such that $\pi$ becomes an isomorphism when restricted to $U - \{u_0\}$ or the following conditions simultaneously hold true:
\begin{enumerate}
\item[$(C1)$] we have the bounds
$$
1 \leq \# \{(a, b, \chi, c, \nu) \in U : \chi \neq 0 \textup{ or } \nu \neq 0\} \leq 2d;
$$
\item[$(C2)$] the set $\{a \in A - \{\textup{id}\} : \exists b \in A \textup{ such that } (a, b) \in \mathcal{J}' - \mathcal{J}''\}$ does not generate $A$;
\item[$(C3)$] all $(a, b, \chi) \in U$ and all $(a, b, \chi, c) \in U$ satisfy $\chi = 0$; 
\item[$(C4)$] all elements $(a_1, b_1, 0, c_1), (a_1, b_1, 0, c_2) \in U$ satisfy $c_1 = c_2$. Furthermore, all elements $(a_1, b_1, 0, c_1, 0), (a_1, b_1, 0, c_2, 0) \in U$ satisfy $c_1 = c_2$.
\end{enumerate}
\end{theorem}

\begin{proof}
It is readily verified that $\mathcal{U}$ and $\mathcal{U}_f$ are unlinked. Let $U$ be an arbitrary unlinked set. Then it suffices to prove that at least one of the following four alternatives hold:
\begin{itemize}
\item[$(A1)$] $U \subseteq \mathcal{U}$, or
\item[$(A2)$] $U \subseteq \mathcal{U}_f$ for some admissible $f$, or
\item[$(A3)$] the conditions $(C1)$, $(C2)$, $(C3)$, $(C4)$ all hold, or
\item[$(A4)$] there exists an element $u_0 \in U$ such that $\pi$ becomes an isomorphism when restricted to $U - \{u_0\}$.
\end{itemize}
To this end, let $u = (a, b, \chi) \in U$ or let $u = (a, b, \chi, c) \in U$ or let $u = (a, b, \chi, c, 0) \in U$. We are first going to show that $\chi = 0$. Let $a' \in A - \{\text{id}\}$ be arbitrary and choose any $b' \in A - \{b\}$. Since $\pi: U \rightarrow (A - \{\text{id}\}) \times A$ is surjective, we may find some element $v \in U$ with $\pi(v) = (a', b')$. By our choice of $b'$, we see that $u \neq v$. Because $u$ and $v$ are unlinked, we deduce that
$$
\chi(a') = 0.
$$
Since $\chi(a') = 0$ holds for every $a' \in A - \{\text{id}\}$, this forces $\chi = 0$, and hence
\begin{equation}
\label{eEasyU}
U \cap \mathcal{I}_3''' \subseteq \{(a, b, 0): a \in A - \{\textup{id}\}, b \in A, (a, b) \in \mathcal{J}' - \mathcal{J}''\}
\end{equation}
and
\begin{equation}
\label{eEasyU2}
U \cap \mathcal{I}_3'' \subseteq \{(a, b, 0, c): a \in A - \{\textup{id}\}, b \in A, c \in \mathbb{F}_p, (a, b) \in \mathcal{J}''\}
\end{equation}
and
\begin{equation}
\label{eGeneralClaim}
(a, b, \chi, c, 0) \in U \Longrightarrow \chi = 0.
\end{equation}
We will now distinguish two cases.

\paragraph{Case 1.} First suppose that 
$$
\# \{(a, b, \chi, c, \nu) \in U : \nu \neq 0\} \leq 2d.
$$
By equation \eqref{eGeneralClaim}, it then follows that
$$
\# \{(a, b, \chi, c, \nu) \in U : \nu \neq 0 \text{ or } \chi \neq 0\} = \# \{(a, b, \chi, c, \nu) \in U : \nu \neq 0\} \leq 2d.
$$
If $\# \{(a, b, \chi, c, \nu) \in U : \nu \neq 0 \text{ or } \chi \neq 0\} = 0$, then, drawing upon our previous results \eqref{eEasyU} and \eqref{eEasyU2}, we conclude that $U \subseteq \mathcal{U}$, so $(A1)$ holds. Thus, it remains to handle the case
\begin{equation}
\label{ePariah}
1 \leq \# \{(a, b, \chi, c, \nu) \in U : \nu \neq 0 \text{ or } \chi \neq 0\} \leq 2d.
\end{equation}
Note that this is precisely the condition $(C1)$. We will now distinguish two subcases, namely whether $(C2)$ holds or not.

\paragraph{Subcase A: the condition $(C2)$ holds.} We will show that $(A3)$ holds in this case. Note that we already know that $(C1)$ and $(C2)$ are true, so it remains to prove that $(C3)$ and $(C4)$ hold. By equations \eqref{eEasyU} and \eqref{eEasyU2} again, we have that $(a, b, \chi) \in U$ and $(a, b, \chi, c) \in U$ implies $\chi = 0$, thus $(C3)$ holds. 

So it remains to prove that all $u_1 = (a_1, b_1, 0, c_1) \in U$ and $u_2 = (a_1, b_1, 0, c_2) \in U$ satisfy $c_1 = c_2$ (and similarly all $u_1 = (a_1, b_1, 0, c_1, 0) \in U$ and $u_2 = (a_1, b_1, 0, c_2, 0) \in U$ satisfy $c_1 = c_2$). To this end, fix some $v := (a, b, \chi, c, \nu) \in U$ with $\nu \neq 0$, which exists by equation \eqref{ePariah}. Then using that $v$ is unlinked with $u_1$ and $u_2$, we obtain
$$
\chi(a_1) + \nu c_1 = \chi(a_1) + \nu c_2 = 0.
$$
Since $\nu \neq 0$, this implies $c_1 = c_2$, and therefore $(C4)$ holds.

\paragraph{Subcase B: the condition $(C2)$ fails.} We will now show that either $(A2)$ or $(A4)$ is true. Using the negation of $(C2)$, it is not difficult to see that all $(a, b, \chi, c, \nu) \in U$ satisfy $\chi = 0$. By equation \eqref{ePariah} and by what we have just shown, there exists some $u_0 = (a, b, 0, c, \nu) \in U$ with $\nu \neq 0$. Recalling that we have already shown the inclusions \eqref{eEasyU} and \eqref{eEasyU2}, we use that $u_0$ is unlinked with every $u \in U - \{u_0\}$ to deduce that the $c$-coordinate of $u$ must also be zero. 

If the $c$-coordinate of $u_0$ is zero as well, then we have $U \subseteq \mathcal{U}_0$ (with $0$ denoting the zero homomorphism, which is always admissible), and thus $(A2)$ is true. If the $c$-coordinate of $u_0$ is nonzero, then it follows from the unlinked assumption that the $\nu$-coordinate of every $u \in U- \{u_0\}$ must also be zero. We conclude that $(A4)$ holds.

\paragraph{Case 2.} Next suppose that 
\begin{equation}
\label{eNonPariahCase}
\# \{(a, b, \chi, c, \nu) \in U : \nu \neq 0\} > 2d.
\end{equation}
We will ultimately show that $(A2)$ holds. In order to do so, we will proceed by making a sequence of claims.

\paragraph{Claim 1 and its consequences.} We claim that for all $u_1 = (a_1, b_1, 0, c_1) \in U$ and $u_2 = (a_1, b_2, 0, c_2) \in U$, we have $c_1 = c_2$ (exactly as in Subcase A). In order to prove this claim, we first fix some $u_{\text{par}, 1} = (a, b, \chi, c, \nu) \in U$ with $\nu \neq 0$, which exists by our assumption \eqref{eNonPariahCase}. Using that $u_{\text{par}, 1}$ is unlinked with both $u_1$ and $u_2$, we obtain
$$
\chi(a_1) + c_1 \nu = \chi(a_1) + c_2 \nu = 0,
$$
which readily gives the claim. 

From Claim 1 and \eqref{eEasyU2}, we deduce that there exists a function $f_0: p_1(\mathcal{I}_3'') \rightarrow \mathbb{F}_p$ such that
\begin{equation}
\label{eUf0}
U \cap \mathcal{I}_3'' \subseteq \{(a, b, 0, f_0(a)) : a \in A - \{\text{id}\}, b \in A\}.
\end{equation}
Observe that $p_1(\mathcal{I}_3'') = p_1(U \cap \mathcal{I}_3'')$ by surjectivity of $\pi$, which we shall use repeatedly. Using this observation and using $u_{\text{par}, 1}$ one more time, we have for all $a \in p_1(\mathcal{I}_3'')$
$$
\chi(a) + f_0(a) \nu = 0.
$$
Taking $\Z$-linear combinations of the above identity and using that $\nu \neq 0$, we conclude that for all $a_1, \dots, a_m \in p_1(\mathcal{I}_3'')$ and all $k_1, \dots, k_m \in \Z$ with $k_1 a_1 + \dots + k_m a_m = 0$
\begin{equation}
\label{eKeyFact}
\sum_{i = 1}^m k_i f_0(a_i) = 0.
\end{equation}
Define $A''$ to be the subgroup generated by $p_1(\mathcal{I}_3'')$. It follows from equation \eqref{eKeyFact} and Lemma \ref{lExtendingF} that there exists a homomorphism $f: A'' \rightarrow \mathbb{F}_p$ such that $f = f_0$ on $p_1(\mathcal{I}_3'')$. In particular, we conclude that equation \eqref{eUf0} may be strengthened to
$$
U \cap \mathcal{I}_3'' \subseteq \{(a, b, 0, f(a)) : a \in A - \{\text{id}\}, b \in A\}.
$$

\paragraph{Claim 2.} We next claim that for all $u_1 = (a_1, b_1, \chi_1, c_1, \nu_1) \in U$ and $u_2 = (a_1, b_2, \chi_2, c_2, \nu_2) \in U$, we have $c_1 = c_2$. In order to prove this claim, we further fix some distinct elements (and distinct from $u_{\text{par}, 1}$)
\begin{align*}
&u_{\text{par}, 2} = (a', b', \chi', c', \nu') \in U \\
&u_{\text{par}, 3} = (a'', b'', \chi'', c'', \nu'') \in U
\end{align*}
with $\nu', \nu''$ both nonzero; such elements exist by our assumption \eqref{eNonPariahCase}. Now given $u_1$ and $u_2$, we fix some $j \in \{1, 2, 3\}$ such that $u_{\text{par}, j} \not \in \{u_1, u_2\}$. Using that $u_{\text{par}, j}$ is unlinked with both $u_1$ and $u_2$, we obtain the claim. Hence there exists a function $g_0: p_1(\mathcal{I}_3') \rightarrow \mathbb{F}_p$ such that
\begin{equation}
\label{eUg0}
U \cap \mathcal{I}_3' \subseteq \{(a, b, \chi, g_0(a), \nu) : a \in A - \{\text{id}\}, b \in A, \chi \in A^\vee, \nu \in \mathbb{F}_p\}.
\end{equation}

\paragraph{Claim 3.} Define $A'$ to be the subgroup generated by $p_1(\mathcal{I}_3')$. We claim that there exists a homomorphism $g: A' \rightarrow \mathbb{F}_p$ with the property that $g = g_0$ on $p_1(\mathcal{I}_3')$. To this end, we extract a minimal generating set $S$ of $A'$ from $p_1(\mathcal{I}_3')$. Note that $|S| \leq d$. Now let $a_1, \dots, a_m \in S$ be distinct, so $m \leq d$. Fix elements $\alpha_1, \dots, \alpha_m \in U \cap \mathcal{I}_3'$ such that $p_1(\alpha_i) = a_i$. Then it follows from equation \eqref{eNonPariahCase} that there exists some $u_{\text{par}} = (a, b, \chi, c, \nu) \not \in \{\alpha_1, \dots, \alpha_m\}$ with $\nu \neq 0$. Using that $u_{\text{par}}$ is unlinked with $\alpha_i$, we deduce that
$$
\chi(a_i) + g_0(a_i) \nu = 0.
$$
This implies that for all distinct $a_1, \dots, a_m \in S$ and all $k_1, \dots, k_m \in \Z$ with $k_1 a_1 + \dots + k_m a_m = 0$ we have
\begin{equation}
\label{eLinearCombination}
\sum_{i = 1}^m k_i g_0(a_i) = 0.
\end{equation}
This gives the existence of our desired $g$ thanks to Lemma \ref{lExtendingF} (note that we have only covered the case that $a_1, \dots, a_m \in S$ are distinct, but this easily implies equation \eqref{eLinearCombination} also when the $a_i$ are not distinct). Hence equation \eqref{eUg0} becomes
$$
U \cap \mathcal{I}_3' \subseteq \{(a, b, \chi, g(a), \nu) : a \in A - \{\text{id}\}, b \in A, \chi \in A^\vee, \nu \in \mathbb{F}_p\}.
$$

\paragraph{Claim 4.} We extract minimal generating sets $S$ and $T$ for respectively $\langle p_1(\mathcal{I}_3') \rangle$ and $\langle p_1(\mathcal{I}_3'') \rangle$. Then we claim that for all distinct $a_1, \dots, a_m \in S$, all $k_1, \dots, k_m \in \Z$, all $b_1, \dots, b_n \in T$ and all $l_1, \dots, l_n \in \Z$ with $k_1 a_1 + \dots + k_m a_m = l_1 b_1 + \cdots + l_n b_n$, we have
$$
\sum_{i = 1}^m k_i f(a_i) = \sum_{j = 1}^n l_j g(b_j).
$$
Fix elements $\alpha_1, \dots, \alpha_m \in U \cap \mathcal{I}_3'$ and $\beta_1, \dots, \beta_n \in U \cap \mathcal{I}_3''$ projecting to respectively $a_1, \dots, a_m$ and $b_1, \dots, b_n$. Since $|S| \leq d$ and $|T| \leq d$, it follows from equation \eqref{eNonPariahCase} that there exists some $u_{\text{par}} = (a, b, \chi, c, \nu) \in U - \{\alpha_1, \dots, \alpha_m, \beta_1, \dots, \beta_n\}$ with $\nu \neq 0$. Using that $u_{\text{par}}$ is unlinked with $\alpha_i$ respectively $\beta_j$, we get the two equations
$$
\chi(a_i) + f(a_i) \nu = 0, \quad \quad \chi(b_j) + g(b_j) \nu = 0.
$$
Then the claim follows upon taking $\Z$-linear combinations of the above identities and using that $\nu \neq 0$.

\paragraph{End of proof.} From Claim 4 and Lemma \ref{lGlue}, it follows that $f$ and $g$ glue to a homomorphism $h: A' + A'' \rightarrow \mathbb{F}_p$, thus we have
\begin{align}
\label{eUh}
&U \cap \mathcal{I}_3'' \subseteq \{(a, b, 0, h(a)) : a \in A - \{\text{id}\}, b \in A\}. \nonumber \\
&U \cap \mathcal{I}_3' \subseteq \{(a, b, \chi, h(a), \nu) : a \in A - \{\text{id}\}, b \in A, \chi \in A^\vee, \nu \in \mathbb{F}_p\}. 
\end{align}
Now fix some $u = (a, b, \chi, h(a), \nu) \in U$ again. If $a \in p_1(\mathcal{I}_3''')$, then it follows that $\chi(a) = 0$ by using that $u$ is unlinked with any $v \in U$ satisfying $p_1(v) = a$. Using equation \eqref{eUh}, we also obtain that $\chi(a) = 0$ if $a \in \ker(h)$. We conclude that $\chi$ vanishes on $p_1(\mathcal{I}_3''')$ and $\ker(h)$. 


There are now two final cases to consider. Firstly, suppose that $p_1(\mathcal{I}_3''')$ and $\ker(h)$ span $A$. Then $(a, b, \chi, h(a), \nu) \in U$ implies $\chi = 0$. In that case we use equation \eqref{eNonPariahCase} and the unlinked property one final time to conclude that $h$ is the trivial character. Writing $\mathcal{U}_0$ for the set $\mathcal{U}_f$ with $f = 0$ (which is admissible), we conclude that $U \subseteq \mathcal{U}_0$, 

Secondly, suppose that $p_1(\mathcal{I}_3''')$ and $\ker(h)$ do not span $A$. Since $p_1(\mathcal{I}_3''')$ contains every element in $A - (A' + A'')$ and $A' + A''$ is a subgroup, this forces $A' + A'' = A$. Moreover, since $\ker(h)$ is an index $p$ subgroup of $A' + A'' = A$, we deduce that $p_1(\mathcal{I}_3''') \subseteq \ker(h)$, and therefore $h$ is admissible. Recalling that $\chi$ vanishes on $\ker(h)$, we see that $\chi = x \cdot h$ for some $x \in \mathbb{F}_p$. Hence
\begin{align*}
&U \cap \mathcal{I}_3'' \subseteq \{(a, b, 0, h(a)) : a \in A - \{\text{id}\}, b \in A\}. \nonumber \\
&U \cap \mathcal{I}_3' \subseteq \{(a, b, x \cdot h, h(a), \nu) : a \in A - \{\text{id}\}, b \in A, x \in \mathbb{F}_p\}. 
\end{align*}
As the final step of the proof, we will now show that for all $u := (a, b, x \cdot h, h(a), \nu) \in U$, we have $\nu = -x$. So let $(a, b, x \cdot h, h(a), \nu) \in U$. Since $\ker(h)$ is an index $p > 2$ subgroup of $A$, we may fix some $a' \in A - \{\text{id}\}$ with $h(a') \neq 0$ and $a' \neq a$. Then we see that $a' \not \in p_1(\mathcal{I}_3''')$ (as otherwise $h(a') = 0$ by the inclusion $p_1(\mathcal{I}_3''') \subseteq \ker(h)$ proven just above). Therefore, by surjectivity of $\pi$, there exist $v := (a', b', 0, h(a')) \in U$ or $v := (a', b', x' \cdot h, h(a'), \nu') \in U$. Using that $u$ is unlinked with $v$, we obtain that $\nu = -x$. Therefore $(A2)$ holds, as desired.
\end{proof}

\section{Oscillation results}
\label{sAnalytic}
\subsection{Large sieve for Dirichlet characters}
We start with a large sieve result. Conveniently, in our application the relevant variables will \emph{always} be coprime. Thus, our large sieve statement involves ordinary Dirichlet characters as opposed to modified Dirichlet characters. We define for a squarefree integer $a$ divisible exclusively by primes $1$ modulo $p^n$ the Dirichlet character
$$
\chi_{a, n} := \prod_{q \mid a} \chi_{q, n}.
$$

\begin{lemma}
\label{lLargeSieve}
Let $p > 2$ be a prime number and let $k_1, k_2 \in \Z_{\geq 1}$. Then there are $C, \delta > 0$ such that for all $(c_1, c_2) \in \Z/p^{k_1}\Z \times \Z/p^{k_2} \Z$ with $(c_1, c_2) \neq (0, 0)$, all $X \geq 100$, all $Y \geq 100$ and all complex coefficients $\beta_b$ with $|\beta_b| \leq 1$
\begin{multline*}
\sum_{\substack{a \leq X \\ q \mid a \Rightarrow q \equiv 1 \bmod p^{k_1}}} \mu^2(a) \left| \sum_{\substack{b \leq Y \\ q \mid b \Rightarrow q \equiv 1 \bmod p^{k_2}}} \beta_b \mu^2(ab) \chi_{a, k_1}^{c_1}(b) \chi_{b, k_2}^{c_2}(a) \right| \\
\leq CXY (X^{-\delta} + Y^{-\delta}) (\log XY)^C.
\end{multline*}
\end{lemma}

\begin{proof}
Our goal will be to apply \cite[Proposition 4.21]{Santens}. In the notation of that proposition, we take $K = \Q(\zeta_{p^{k_1}})$ and $L = \Q(\zeta_{p^{k_2}})$. We let $U$ be the downward-closed set (see \cite[Definition 4.14]{Santens}) consisting of the ideals in $(I, J) \in I_K \times I_L$ of norm bounded by respectively $X$ and $Y$. We take
$$
\alpha(\mathfrak{m}, \mathfrak{n}) = \left(\frac{N_{L/\Q}(\mathfrak{n})}{\mathfrak{m}}\right)_{\Q(\zeta_{p^{k_1}}), p^{k_1}}^{c_1} \left(\frac{N_{K/\Q}(\mathfrak{m})}{\mathfrak{n}}\right)_{\Q(\zeta_{p^{k_2}}), p^{k_2}}^{c_2},
$$
where $(\cdot/\cdot)_{M, n}$ denotes the $n$-th power residue symbol in a number field $M$ (which only makes sense if $M$ contains the $n$-th roots of unity). We take $q = p^3$ and we take $A$ a large number in terms of $p$, $k_1$ and $k_2$ only. It is then readily verified that $\alpha(\mathfrak{m}, \mathfrak{n})$ is an $(A, q)$-oscillating bilinear character in the sense of \cite[Definition 4.19]{Santens}. Now \cite[Proposition 4.21]{Santens} provides us with $C, \delta > 0$ and the estimate
\begin{equation}
\label{eSantensLS}
\left| \sum_{N_{K/\Q}(\mathfrak{m}) \leq X} \sum_{N_{L/\Q}(\mathfrak{n}) \leq Y} a_{\mathfrak{m}} b_{\mathfrak{n}} \alpha(\mathfrak{m}, \mathfrak{n}) \right| \leq CXY (X^{-\delta} + Y^{-\delta}) (\log XY)^C
\end{equation}
valid for all complex numbers $a_{\mathfrak{m}}$ and $b_{\mathfrak{n}}$ of magnitude at most $1$.

We claim that, for a good choice of $a_{\mathfrak{m}}$ and $b_{\mathfrak{n}}$, equation \eqref{eSantensLS} becomes exactly the stated bound in our lemma. First, we observe that by the standard conventions for Dirichlet character, we may replace $\mu^2(ab)$ by $\mu^2(b)$. Hence, for a good choice of $\beta_a'$, we have
\begin{multline*}
\sum_{\substack{a \leq X \\ q \mid a \Rightarrow q \equiv 1 \bmod p^{k_1}}} \mu^2(a) \left| \sum_{\substack{b \leq Y \\ q \mid b \Rightarrow q \equiv 1 \bmod p^{k_2}}} \beta_b \mu^2(ab) \chi_{a, k_1}^{c_1}(b) \chi_{b, k_2}^{c_2}(a) \right| = \\ \sum_{\substack{a \leq X \\ q \mid a \Rightarrow q \equiv 1 \bmod p^{k_1}}} \sum_{\substack{b \leq Y \\ q \mid b \Rightarrow q \equiv 1 \bmod p^{k_2}}} \mu^2(a) \mu^2(b) \beta_a' \beta_b \chi_{a, k_1}^{c_1}(b) \chi_{b, k_2}^{c_2}(a).
\end{multline*}
Now, to relate the above sum with \eqref{eSantensLS}, we observe that for every prime $q \equiv 1 \bmod p^k$, there exists precisely one prime ideal $\mathfrak{q}$ of $\Z[\zeta_{p^k}]$ above $q$ such that
$$
\chi_{q, k}(a) = \left(\frac{a}{\mathfrak{q}}\right)_{\Q(\zeta_{p^k}), p^k}.
$$
We will call this ideal $\mathrm{Up}_k(q)$, and we define $\mathrm{Up}_k(n) = \prod_{q \mid n} \mathrm{Up}_k(q)$ for all squarefree integers $n$ exclusively divisible by primes $1$ modulo $p^k$. Hence we have
$$
\sum_{\substack{a \leq X \\ q \mid a \Rightarrow q \equiv 1 \bmod p^{k_1}}} \sum_{\substack{b \leq Y \\ q \mid b \Rightarrow q \equiv 1 \bmod p^{k_2}}} \hspace{-0.5cm} \mu^2(a) \mu^2(b) \beta_a' \beta_b \chi_{a, k_1}^{c_1}(b) \chi_{b, k_2}^{c_2}(a) = 
\sum_{N_{K/\Q}(\mathfrak{m}) \leq X} \sum_{N_{L/\Q}(\mathfrak{n}) \leq Y} a_{\mathfrak{m}} b_{\mathfrak{n}} \alpha(\mathfrak{m}, \mathfrak{n}),
$$
where we have chosen 
$$
a_{\mathfrak{m}} =
\begin{cases}
\beta_{N_{K/\Q}(\mathfrak{m})}' &\text{if } \mathfrak{m} \in \im(\text{Up}_{k_1}) \\
0 &\text{if } \mathfrak{m} \not \in \im(\text{Up}_{k_1})
\end{cases}
, \quad \quad \quad
b_{\mathfrak{n}} = 
\begin{cases}
\beta_{N_{L/\Q}(\mathfrak{n})} &\text{if } \mathfrak{n} \in \im(\text{Up}_{k_2}) \\
0 &\text{if } \mathfrak{n} \not \in \im(\text{Up}_{k_2})
\end{cases}
.
$$
This gives the lemma.
\end{proof}

We shall later need to apply Lemma \ref{lLargeSieve} in hyperbolic regions instead of boxes. Standard arguments allow one to cover such a hyperbolic region by boxes, so we may state the following corollary.

\begin{corollary}
\label{cLargeSieve}
Let $r \geq 2$ be an integer. Let $p > 2$ be a prime number and let $k_1, k_2 \in \Z_{\geq 1}$. Then there are $C, \delta > 0$ such that the following holds. Let $(c_1, c_2) \in \Z/p^{k_1}\Z \times \Z/p^{k_2} \Z$ with $(c_1, c_2) \neq (0, 0)$. Let $\alpha, \beta: \Z_{\geq 1}^{r - 1} \rightarrow \mathbb{C}$ satisfy $|\alpha(\mathbf{n})|, |\beta(\mathbf{n})| \leq 1$ and assume that $\alpha$ and $\beta$ are supported on squarefree integers. Then we have for all $X, z \geq 100$
\begin{multline*}
\left|\sum_{\substack{\mathbf{m} \in \Z_{\geq 1}^r, \ m_1 \cdots m_r \leq X, \ m_1, m_2 > z \\ q \mid m_1 \Rightarrow q \equiv 1 \bmod p^{k_1} \\ q \mid m_2 \Rightarrow q \equiv 1 \bmod p^{k_2}}} \chi_{m_1, k_1}^{c_1}(m_2) \chi_{m_2, k_2}^{c_2}(m_1) \alpha(m_1, m_3, \ldots, m_r) \beta(m_2, m_3, \ldots, m_r) \right| \\
\leq \frac{CX (\log X)^C}{z^\delta}.
\end{multline*}
\end{corollary}

\begin{proof}
This follows from Lemma \ref{lLargeSieve} by a standard argument, see for example \cite[Corollary 4.4]{KPS}.
\end{proof}

\subsection{A Siegel--Walfisz type result}
Throughout this subsection, we fix an odd prime number $p$. Recall that we have fixed a choice of Dirichlet character $\chi_q: \Z \rightarrow \mathbb{C}$ with period $q$ and of order $p$ for each prime $q \equiv 1 \bmod p$. This choice is not intrinsic, as $\chi_q^c$ has the same field of definition for every invertible $c \in (\Z/p\Z)^\ast$. Nevertheless, we have already encountered sums of the type
\begin{align}
\label{eBiasedSum}
\sum_{\substack{q \leq X \\ q \equiv 1 \bmod p}} \chi_q(\ell)
\end{align}
for some fixed prime $\ell$. Although it is intrinsic whether $\chi_q(\ell)$ equals $1$ or not, the sum \eqref{eBiasedSum} can easily be made to be biased (for example, by choosing the $\chi_q$ to satisfy $\chi_q(\ell) \in \{1, \zeta\}$ for some fixed primitive $p$-th root of unity $\zeta \in \mathbb{C}$). We wish to avoid this behavior, which motivates our next two definitions and theorem. We start with some notation.

\begin{mydef}
Let $\mathcal{I}$ be an index set. We say that a sequence of integers $(a_i)_{i \in \mathcal{I}}$ and a vector $(c_i)_{i \in \mathcal{I}} \in \mathbb{F}_p^{\mathcal{I}}$ is eligible if the $a_i \geq 1$ are squarefree, pairwise coprime, exclusively divisible by primes $0, 1$ modulo $p$ and if there exists $i \in \mathcal{I}$ such that $a_i > 1$ and $c_i \neq 0$.
\end{mydef}

Our next definition encapsulates the potentially bad behavior in the sum \eqref{eBiasedSum} that we wish to avoid in general.

\begin{mydef}
Let $X > 10$ be a real number, let $C > 0$ be a real number and let $n \in \Z_{\geq 1}$. We say that the characters $(\chi_q)_{q \leq X}$ are prepared for equidistribution up to $(C, n)$ if for all eligible sequences $(a_i)_{1 \leq i \leq n}$ and $(c_i)_{1 \leq i \leq n} \in \mathbb{F}_p^n$ with $1 \leq a_i \leq (\log X)^C$, for all $\zeta \in \langle \zeta_p \rangle$ with $\zeta \neq 1$ and for all $(\zeta_i)_{1 \leq i \leq n}$ with $\zeta_i \in \langle \zeta_p \rangle$ we have
$$
\left| \#\left\{\begin{array}{c} q \leq X, \\ q \equiv 1 \bmod p \end{array} \hspace{-0.15cm} : \hspace{-0.15cm} \begin{array}{c} \prod_{i = 1}^n \chi_q^{c_i}(a_i) = \zeta \\ \chi_{a_i}(q) = \zeta_i \end{array} \right\}  - \frac{\#\left\{\begin{array}{c} q \leq X, \\ q \equiv 1 \bmod p \end{array} \hspace{-0.15cm} : \hspace{-0.15cm} \begin{array}{c} \prod_{i = 1}^n \chi_q^{c_i}(a_i) \neq 1 \\ \chi_{a_i}(q) = \zeta_i \end{array}\right\}}{p - 1} \right| \leq X^{3/4}.
$$
\end{mydef}

\begin{theorem}
\label{tGoodChoice}
Let $C > 0$ be a real number and let $n \in \Z_{\geq 1}$. Then there exists a real number $X_0 > 10$ such that for all $X \geq X_0$ there exists a choice of characters $(\chi_q')_{q \leq X}$ that are prepared for equidistribution up to $(C, n)$.
\end{theorem}

\begin{proof}
Write $\mathcal{P}_{p, X}$ for the set of primes $q \leq X$ with $q \equiv 1 \bmod p$, $\prod_{i = 1}^n \chi_q^{c_i}(a_i) \neq 1$ and $\chi_{a_i}(q) = \zeta_i$ for all $1 \leq i \leq n$. Write $L(a_i)$ for the extension cut out by $\chi_{a_i}$. Since $p$ is odd, the extensions $\Q(\zeta_p) L(a_i)/\Q(\zeta_p)$ and $\Q(\zeta_p, \sqrt[p]{a_i})/\Q(\zeta_p)$ are linearly disjoint. Hence it follows from the effective Chebotarev density theorem and the eligible assumption that
\begin{align}
\label{ePpx}
\# \mathcal{P}_{p, X} \gg_{C, n, p} X/\log X.
\end{align}
We fix one choice of $\chi_q$ for each $q \in \mathcal{P}_{p, X}$, and we define
$$
\mathcal{S} := \mathrm{Map}(\mathcal{P}_{p, X}, (\Z/p\Z)^\ast).
$$
Then the set of all possible choices for $\chi_q$ is in bijection with $\mathcal{S}$ by attaching the choice $(\chi_q^{f(q)})_{q \in \mathcal{P}_{p, X}}$ to a map $f \in \mathcal{S}$. We endow $(\Z/p\Z)^\ast$ with the uniform probability measure, which naturally gives rise to a probability measure on $\mathcal{S}$ (with evaluation at different primes of $\mathcal{P}_{p, X}$ defining independent events). 

We now apply Hoeffding's inequality to the independent random variables $X_q = f(q)$ with $f \in \mathcal{S}$. The total number of independent random variables is $\# \mathcal{P}_{p, X}$, for which we recall the lower bound from equation \eqref{ePpx}. If the $(a_i)_{1 \leq i \leq n}$, $(c_i)_{1 \leq i \leq n}$, $\zeta$ and $(\zeta_i)_{1 \leq i \leq n}$ are fixed and if $X$ is sufficiently large in terms of $C$, $n$ and $p$ only, then Hoeffding's inequality shows that the number of choices $f \in \mathcal{S}$ not satisfying
$$
\left| \#\left\{q \in \mathcal{P}_{p, X} : \prod_{i = 1}^n \chi_q^{c_i f(q)}(a_i) = \zeta\right\}  - \frac{\# \mathcal{P}_{p, X}}{p - 1} \right| \leq X^{3/4}
$$
is bounded by $|\mathcal{S}| \exp(-X^{1/10})$. Therefore the total amount of choices not prepared for equidistribution up to $(C, n)$ is bounded by
$$
(p - 1) p^{2n} (\log X)^{Cn} |\mathcal{S}| \exp(-X^{1/10})
$$
as there are $(p - 1) p^{2n}$ choices for the $c_i$, $\zeta_i$ and $\zeta$, and $(\log X)^{C n}$ choices for the $a_i$. Since this is smaller than $|\mathcal{S}|$ for sufficiently large $X$, the theorem follows.
\end{proof}

We will henceforth always pick our characters to be prepared for equidistribution up to $(C, n)$ for some fixed large $C$ and large $n$, which is possible thanks to Theorem \ref{tGoodChoice}. With this problem out of the way, we can show equidistribution of the sum \eqref{eBiasedSum}.

\begin{mydef}
Let $\mathcal{I}$ be an index set. We say that a sequence of integers $(a_i)_{i \in \mathcal{I}}$ and a vector $((c_i, d_i))_{i \in \mathcal{I}} \in \left(\mathbb{F}_p^2\right)^{\mathcal{I}}$ is quite eligible if the $a_i \geq 1$ are squarefree, pairwise coprime, exclusively divisible by primes $0, 1$ modulo $p$ and there exists $i \in \mathcal{I}$ such that $a_i > 1$ and $(c_i, d_i) \neq (0, 0)$.
\end{mydef}

\begin{lemma}
\label{lSW}
Let $C_1 > 0$ be a real number, let $\mathcal{I}$ be a finite index set and let $(w_i)_{i \in \mathcal{I}}$ be positive real numbers. Then there exists $C_2 > 0$ such that for all $d \leq X$, for all quite eligible sequences $(a_i)_{i \in \mathcal{I}}$ with $1 \leq a_i \leq (\log X)^{C_1}$ and for all $(c_i, d_i) \in \mathbb{F}_p^2$
$$
\left| \sum_{\substack{n \leq X \\ q \mid n \Rightarrow q \equiv 1 \bmod p \\ \gcd(n, d) = 1}} \mu^2(n) \prod_{i \in \mathcal{I}} \frac{\chi_n^{c_i}(a_i) \chi_{a_i}^{d_i}(n)}{w_i^{\omega(n)}} \right| \leq \frac{C_2X}{(\log X)^{C_1}}.
$$
\end{lemma}

\begin{proof}
We will prove this when $n$ is restricted to primes $q$, i.e.
$$
\left| \sum_{\substack{q \leq Q, q \equiv 1 \bmod p \\ q \nmid d}} \prod_{i \in \mathcal{I}} \frac{\chi_q^{c_i}(a_i) \chi_{a_i}^{d_i}(q)}{w_i} \right| \leq \frac{C_2 Q}{(\log Q)^{C_1}}
$$
for $Q > \exp\left((\log X)^{1/100}\right)$. Once we establish such an estimate, the lemma follows from the LSD method as codified in \cite[Theorem 13.2]{Kou}. 

Using that $d \leq X$, it is straightforward to estimate the contribution from $q \mid d$. Since $w_i$ is fixed, it therefore suffices to show that for all $C_1 > 0$ and for all $\mathcal{I}$, there exists $C_2 > 0$ such that for all eligible $(a_i)_{i \in \mathcal{I}}$ with $1 \leq a_i \leq (\log X)^{C_1}$ and for all $(c_i, d_i) \in \mathbb{F}_p^2$
\begin{align}
\label{eSWRed}
\left| \sum_{\substack{q \leq Q \\ q \equiv 1 \bmod p}} \prod_{i \in \mathcal{I}} \chi_q^{c_i}(a_i) \chi_{a_i}^{d_i}(q) \right| \leq \frac{C_2 Q}{(\log Q)^{C_1}}.
\end{align}
If $c_i = 0$ for all $i \in \mathcal{I}$, then we immediately get the desired cancellation from the Siegel--Walfisz theorem and the quite eligible assumption. Henceforth we assume that $c_i \neq 0$ for at least one $i \in \mathcal{I}$.

At this stage, we will transfer our sum to $\Q(\zeta_p)$. Note that every prime $q \equiv 1 \bmod p$ splits completely in $\Q(\zeta_p)$, so there are exactly $p - 1$ primes $\mathfrak{q}$ above $q$. Among these primes, we recall that there is precisely one prime $\mathfrak{q}$, namely $\mathrm{Up}(q)$, such that
$$
\chi_q(a) = \left(\frac{a}{\mathfrak{q}}\right)_{\Q(\zeta_p), p}
$$
for all $a \in \Z$. Then the sum in equation \eqref{eSWRed} is equal to
$$
\sum_{\substack{q \leq Q \\ q \equiv 1 \bmod p}} \prod_{i \in \mathcal{I}} \chi_q^{c_i}(a_i) \chi_{a_i}^{d_i}(q) = \sum_{\substack{q \leq Q \\ q \equiv 1 \bmod p}} \prod_{i \in \mathcal{I}} \left(\frac{a_i}{\mathrm{Up}(q)}\right)_{\Q(\zeta_p), p}^{c_i} \left(\frac{q}{\mathrm{Up}(a_i)}\right)_{\Q(\zeta_p), p}^{d_i}.
$$
Inspired by this, we define for each vector $\mathbf{e} = (e_i)_{i \in \mathcal{I}} \in \mathbb{F}_p^{\mathcal{I}}$ the Hecke character $\rho_{\mathbf{e}}$ on $\Q(\zeta_p)$ given by
$$
\rho_{\mathbf{e}}(\mathfrak{q}) = \prod_{i \in \mathcal{I}} \left(\frac{a_i}{\mathfrak{q}}\right)_{\Q(\zeta_p), p}^{c_i} \left(\frac{N(\mathfrak{q})}{\mathrm{Up}(a_i)}\right)_{\Q(\zeta_p), p}^{e_i}.
$$
We claim that $\rho_{\mathbf{e}}$ is not the trivial character. To this end, first recall that $c_i \neq 0$ for at least one $i$ and write $L(a_i)$ for the fixed field of $\chi_{a_i}$. Since $p$ is odd, we observe that the extensions $\Q(\zeta_p) L(a_i)/\Q(\zeta_p)$ and $\Q(\zeta_p, \sqrt[p]{a_i})/\Q(\zeta_p)$ are linearly disjoint. Then our quite eligible assumption readily implies that $\rho_{\mathbf{e}}$ is not the trivial character.

The main theorem of Goldstein \cite{Goldstein} gives the existence of $C_2 > 0$ such that for all $\mathbf{e}$ and for all $Q > \exp\left((\log X)^{1/100}\right)$
$$
\left|\sum_{N(\mathfrak{q}) \leq Q} \rho_{\mathbf{e}}(\mathfrak{q})\right| \leq \frac{C_2 Q}{(\log Q)^{C_1}}.
$$
After enlarging $C_2$, the above estimate continues to hold when restricted to completely split primes. This implies that for each vector $(\zeta_i)_{1 \leq i \leq n}$
\begin{align}
\label{eGoldsteinSW}
\left|\sum_{\substack{N(\mathfrak{q}) \leq Q \\ \mathfrak{q} \text{ unr.~of degree 1}}} \rho_{\mathbf{0}}(\mathfrak{q}) \mathbf{1}_{\left(\frac{N(\mathfrak{q})}{\mathrm{Up}(a_i)}\right)_{\Q(\zeta_p), p} = \zeta_i}\right| \leq \frac{C_2 Q}{(\log Q)^{C_1}}.
\end{align}
By grouping together $\mathfrak{q}$ with all of its conjugates (of which there are $p - 1$ in total), equation \eqref{eGoldsteinSW} may be equivalently phrased as
\begin{multline}
\label{eGoldsteinSW2}
\left|(p - 1) \#\left\{\begin{array}{c} q \leq Q, \\ q \equiv 1 \bmod p \end{array} \hspace{-0.15cm} : \hspace{-0.15cm} \begin{array}{c} \prod_{i = 1}^n \chi_q^{c_i}(a_i) = 1 \\ \chi_{a_i}(q) = \zeta_i \end{array} \right\}  - \#\left\{\begin{array}{c} q \leq Q, \\ q \equiv 1 \bmod p \end{array} \hspace{-0.15cm} : \hspace{-0.15cm} \begin{array}{c} \prod_{i = 1}^n \chi_q^{c_i}(a_i) \neq 1 \\ \chi_{a_i}(q) = \zeta_i \end{array}\right\} \right| \\
\leq \frac{C_2 Q}{(\log Q)^{C_1}}.
\end{multline}
Indeed, if there is one conjugate $\mathfrak{q}'$ of $\mathfrak{q}$ with $\rho_{\mathbf{0}}(\mathfrak{q}') = 1$, then all conjugates $\mathfrak{q}'$ of $\mathfrak{q}$ satisfy $\rho_{\mathbf{0}}(\mathfrak{q}') = 1$. In this case we get a total contribution of $p - 1$, and we have $\prod_{i = 1}^n \chi_q^{c_i}(a_i) = 1$. Otherwise, if there is no conjugate $\mathfrak{q}'$ of $\mathfrak{q}$ with $\rho_{\mathbf{0}}(\mathfrak{q}') = 1$, then $\rho_{\mathbf{0}}(\mathfrak{q}')$ runs through all primitive $p$-th roots of unity as $\mathfrak{q}'$ runs through all conjugates. These roots of unity sum together to $-1$, and this case occurs if and only if $\prod_{i = 1}^n \chi_q^{c_i}(a_i) \neq 1$.

Since the ideals $\mathrm{Up}(q)$ have been prepared for equidistribution, we deduce from equation \eqref{eGoldsteinSW2} and the triangle inequality that
$$
\left|\#\left\{\begin{array}{c} q \leq Q, \\ q \equiv 1 \bmod p \end{array} \hspace{-0.24cm} : \hspace{-0.24cm} \begin{array}{c} \prod_{i = 1}^n \chi_q^{c_i}(a_i) = 1 \\ \chi_{a_i}(q) = \zeta_i \end{array} \right\}  - \#\left\{\begin{array}{c} q \leq Q, \\ q \equiv 1 \bmod p \end{array} \hspace{-0.24cm} : \hspace{-0.24cm} \begin{array}{c} \prod_{i = 1}^n \chi_q^{c_i}(a_i) = \zeta \\ \chi_{a_i}(q) = \zeta_i \end{array}\right\} \right| \\
\leq \frac{2C_2 Q}{(\log Q)^{C_1}}
$$
for each $\zeta \neq 1$. By freezing $\chi_{a_i}(q)$ and then summing over the possible values of $\prod_{i = 1}^n \chi_q^{c_i}(a_i)$, we conclude that
$$
\left| \sum_{\substack{q \leq Q \\ q \equiv 1 \bmod p}} \prod_{i \in \mathcal{I}} \chi_q^{c_i}(a_i) \chi_{a_i}^{d_i}(q) \right| \leq \frac{2C_2 p^{n + 1} Q}{(\log Q)^{C_1}}
$$
for $Q$ sufficiently large. We recognize the estimate \eqref{eSWRed}, ending the proof of the lemma.
\end{proof}

\section{Proof of main theorem}
\label{sMain}
The goal of this section is twofold. We will finish the proof of Theorem \ref{tMainAnalytic} in our first subsection. Once we have accomplished this goal, we have all the ingredients to prove our main theorem.

\subsection{Main analytic theorem}
Before we start the proof of our main analytic result (Theorem \ref{tMainAnalytic}), we require one more technical result. 

\begin{theorem}
\label{twabBig}
Let $p$ be an odd prime. Let $A$ be a nontrivial finite abelian $p$-group. Then we have for all $s \in A - \{\textup{id}\}$ and all $t \in A$
$$
\lim_{X \rightarrow \infty} \frac{\#\{\varphi \in \mathcal{F}(X) : \textup{there is no } q > \exp((\log X)^{1/2}) \textup{ with } \varphi(\sigma_q) = s, \varphi(\Frob_q) = t\}}{\# \mathcal{F}(X)} = 0.
$$
\end{theorem}

\begin{proof}
By Theorem \ref{tMainPar} we may switch from $\varphi \in \mathcal{F}(X)$ to tuples $\mathbf{w} \in \mathcal{A}$. Let $s \in A - \{\text{id}\}$ and $t \in A$ be given. It is easy to show that almost all $\mathbf{w}$ satisfy $w_a > \exp((\log X)^{3/4})$ for all $a \in A - \{\text{id}\}$. Define $G(X) := \exp((\log X)^{1/2})$. Now among such tuples, we see that the numerator equals
$$
\sum_{\substack{\mathbf{w} \in \mathcal{A} \\ \prod_{a \in A - \{\text{id}\}} w_a \leq X \\ w_a > \exp((\log X)^{3/4}) \forall a \in A - \{\text{id}\}}} \prod_{\substack{q \mid w_s \\ q > G(X)}} \left(1 - \mathbf{1}_{\mathrm{Par}(\mathbf{w})(\Frob_q) = t}\right),
$$
which may be re-expressed as
$$
\sum_{\substack{\mathbf{w} \in \mathcal{A} \\ \prod_{a \in A - \{\text{id}\}} w_a \leq X \\ w_a > \exp((\log X)^{3/4}) \forall a \in A - \{\text{id}\}}} \prod_{\substack{q \mid w_s \\ q > G(X)}} \hspace{-0.15cm} \left(\frac{|A| - 1}{|A|} - \frac{1}{|A|} \sum_{\chi \in A^\vee - \{\text{id}\}} \iota\left(\chi(\mathrm{Par}(\mathbf{w})(\Frob_q) - t)\right)\right).
$$
We now introduce
$$
w_s^{\text{rough}} := \prod_{\substack{q \mid w_s \\ q > G(X)}} q, \quad \quad \quad \quad \omega_{\text{rough}}(n) := \left|\left\{q \mid n : q > G(X)\right\}\right|.
$$
We uniquely factor $w_s = w_s^{\text{rough}} w_s^{\text{smooth}}$. This allows us to rewrite the sum as
\begin{multline*}
\sum_{\substack{\mathbf{w} \in \mathcal{A} \\ \prod_{a \in A - \{\text{id}\}} w_a \leq X \\ w_a > \exp((\log X)^{3/4}) \forall a \in A - \{\text{id}\}}} \sum_{d \mid w_s^{\text{rough}}} \left(\frac{|A| - 1}{|A|}\right)^{\omega_{\text{rough}}(w_s/d)} \left(\frac{-\iota(-t)}{|A|}\right)^{\omega(d)} \\
\sum_{\substack{(\chi_q)_{q \mid d} \\ \chi_q \in A^\vee - \{\text{id}\}}} \prod_{q \mid d} \iota\left(\chi_q\left(\mathrm{Par}(\mathbf{w})(\Frob_q)\right)\right),
\end{multline*}
We change variables in the sum by setting $e = w_s^{\text{rough}}/d$ and moreover by setting for each $\chi \in A^\vee - \{\text{id}\}$
$$
d_\chi = \prod_{\substack{q \mid d \\ \chi_q = \chi}} q.
$$
This defines a change of variables, and transforms the sum into
\begin{multline*}
\sum_{\substack{\mathbf{w} \in \mathcal{A} \\ \prod_{a \in A - \{\text{id}\}} w_a \leq X \\ w_a > \exp((\log X)^{3/4}) \forall a \in A - \{\text{id}\}}} \sum_{\substack{w_s = w_s^{\text{smooth}} e \prod_\chi d_\chi \\ q \mid w_s^{\text{smooth}} \Leftrightarrow q \leq G(X)}} \left(\frac{|A| - 1}{|A|}\right)^{\omega(e)} \left(\frac{-\iota(-t)}{|A|}\right)^{\omega(\prod_{\chi \in A^\vee - \{\text{id}\}} d_\chi)} \\
\prod_{\chi \in A^\vee - \{\text{id}\}} \prod_{q \mid d_\chi} \iota\left(\chi\left(\mathrm{Par}(\mathbf{w})(\Frob_q)\right)\right).
\end{multline*}
We split this sum into $2^{|A| - 1}$ subsums, depending on whether $d_\chi > 1$ or $d_\chi = 1$ for each $\chi \in A^\vee - \{\text{id}\}$. The subsums for which $d_\chi > 1$ for at least one $\chi$ contribute negligibly (compared to $\# \mathcal{F}(X)$) due to the large sieve in Corollary \ref{cLargeSieve}. Here we use in an essential way that $d_\chi > 1$ implies $d_\chi > G(X)$ by construction of $d_\chi$, hence Corollary \ref{cLargeSieve} applies with $z = G(X)$. In the remaining subsum, we have $d_\chi = 1$ for all $\chi$, so this subsum equals
$$
\sum_{\substack{\mathbf{w} \in \mathcal{A} \\ \prod_{a \in A - \{\text{id}\}} w_a \leq X \\ w_a > \exp((\log X)^{3/4}) \forall a \in A - \{\text{id}\}}} \sum_{\substack{w_s = w_s^{\text{smooth}} e \\ q \mid w_s^{\text{smooth}} \Leftrightarrow q \leq G(X)}} \left(\frac{|A| - 1}{|A|}\right)^{\omega(e)}.
$$
We now drop the condition $w_a > \exp((\log X)^{3/4})$ again. Then the resulting sum may be rewritten as
$$
\sum_{n \leq X} f(n),
$$
where $f$ is the multiplicative function, supported on squarefree integers, given by
$$
f(q) =
\begin{cases}
|A| - 1 &\text{if } q = p \\
|\{a \in A - \{\text{id}\} : q \equiv 1 \bmod \ord(a)\}| &\text{if } q \leq G(X), q \neq p \\
\mathbf{1}_{q \equiv 1 \bmod \ord(s)} \left(\frac{|A| - 1}{|A|}\right) + \sum_{\substack{a \in A - \{\text{id}, s\} \\ q \equiv 1 \bmod \ord(a)}} 1 &\text{if } q > G(X).
\end{cases}
$$
Then the theorem follows upon comparing the upper bound coming from Shiu's result \cite{Shiu} with M\"aki's theorem \cite{Maki} to control $\# \mathcal{F}(X)$.
\end{proof}

We are now ready to finish the proof of our main analytic result, which is Theorem \ref{tMainAnalytic}.

\begin{proof}[Proof of Theorem \ref{tMainAnalytic}]
By Theorem \ref{twabBig} and the union bound, it suffices to show that
\begin{equation}
\label{eTrivialSel}
\lim_{X \rightarrow \infty} \frac{\# \{\varphi \in \mathcal{F}^\sharp(X) : \mathrm{Sel}(g(\varphi)) = 0\}}{\# \mathcal{F}(X)} = 1.
\end{equation}
To this end, we are going to show that
$$
\sum_{\varphi \in \mathcal{F}^\sharp(X)} \# \mathrm{Sel}(g(\varphi)) \sim \# \mathcal{F}(X),
$$
which certainly implies equation \eqref{eTrivialSel}. We start by applying Theorem \ref{tCharSum}, which yields
\begin{multline}
\label{eMainSum}
\sum_{\varphi \in \mathcal{F}^\sharp(X)} \# \mathrm{Sel}(g(\varphi)) = \sum_{\substack{\mathbf{w} = (w_i)_{i \in \mathcal{I}_3} \\ \mathbf{f} = (f_a)}}^{\flat \flat \flat} \sum_{\alpha \in \mathbb{F}_p} \tilde{h}(\mathbf{w}, \alpha) \times  \\
\prod_{\substack{k \in \mathcal{I}_3' \\ \ell \in \mathcal{I}_3'}} \psi_{w_k}^{c(k)}\left(w_\ell^{\nu(\ell)}\right) \prod_{\substack{k \in \mathcal{I}_3'' \\ \ell \in \mathcal{I}_3'}} \psi_{w_k}^{c(k)}\left(w_\ell^{\nu(\ell)}\right) \prod_{\ell \in \mathcal{I}_3} \prod_{k \in \mathcal{I}_3} \prod_{q \mid w_k} \chi_{q, \log_p \ord(\chi(\ell))}(w_\ell)^{\chi(\ell)(a(k))} \times \\
\prod_{\ell \in \mathcal{I}_3} \prod_{a \in A - \{\textup{id}\}} \chi_{f_a, \log_p \ord(\chi(\ell))}(w_\ell)^{\chi(\ell)(a)}. 
\end{multline}
We dissect this sum into $2^{|\mathcal{I}_3|}$ pieces: namely, for each subset $U$ of $\mathcal{I}_3$, we define $N(X; U)$ to be the subsum of the RHS, where we impose the additional condition
$$
w_i > \exp\left((\log X)^{\frac{1}{10 p^3 |A|^5}}\right) \Longleftrightarrow i \in U.
$$
Our goal is now to upper bound each individual $N(X; U)$ (except for $U = \mathcal{U}$, which will give rise to the main term).

Note that the last summation condition in equation \eqref{eSummationConditions} ensures that, for sufficiently large $X$, either $N(X; U)$ is empty or that the natural projection map $\pi: U \rightarrow (A - \{\text{id}\}) \times A$ is surjective. Since $N(X; U) = 0$ in the first case, we shall henceforth assume that $\pi: U \rightarrow (A - \{\text{id}\}) \times A$ is surjective. 

We will now distinguish three further cases. First suppose that $U$ is not unlinked. Then we claim that
\begin{equation}
\label{eLSclaim}
N(X; U) = o\left(\#\mathcal{F}(X)\right).
\end{equation}
We will apply Corollary \ref{cLargeSieve} as follows. Since $U$ is not unlinked, there exist two distinct linked indices $u_1, u_2 \in U$. In the notation of Corollary \ref{cLargeSieve}, we take $r := \#\mathcal{I}_3$ and $k_i := \log_p \ord(a(u_i))$ (which depends only on the abelian group $A$). We let $m_1$ be the variable $w_{u_1}$ and we let $m_2$ be the variable $w_{u_2}$, and we let $m_3, \dots, m_r$ be the remaining variables in $\mathcal{I}_3 - \{u_1, u_2\}$. By definition of $U$, we have 
$$
w_{u_1}, w_{u_2} > \exp\left((\log X)^{\frac{1}{10 p^3 |A|^5}}\right) =: z.
$$
Since $z$ is much larger than any fixed power of $\log X$, we obtain equation \eqref{eLSclaim}, once we verify that our sum \eqref{eMainSum} has the shape in Corollary \ref{cLargeSieve}. In particular, we will pay attention to the key hypothesis in Corollary \ref{cLargeSieve} that $(c_1, c_2) \neq (0, 0)$.

Since $u_1$ and $u_2$ are distinct, $w_{u_1}$ and $w_{u_2}$ are coprime by the summation conditions. Hence we have $\psi_{w_{u_1}}(w_{u_2}) = \chi_{w_{u_1}, 1}(w_{u_2})$. Any term not containing both $w_{u_1}$ and $w_{u_2}$ can be absorbed into the $\alpha$ and $\beta$ terms of Corollary \ref{cLargeSieve}. The terms involving both $w_{u_1}$ and $w_{u_2}$ are precisely the product of the following six terms:
\begin{align*}
&\chi_{w_{u_1}, 1}(w_{u_2})^{\mathbf{1}_{u_1, u_2 \in \mathcal{I}_3'} \cdot c(u_1) \nu(u_2)} \\
&\chi_{w_{u_2}, 1}(w_{u_1})^{\mathbf{1}_{u_1, u_2 \in \mathcal{I}_3'} \cdot c(u_2) \nu(u_1)} \\
&\chi_{w_{u_1}, 1}(w_{u_2})^{\mathbf{1}_{u_1 \in \mathcal{I}_3'', u_2 \in \mathcal{I}_3'} \cdot c(u_1) \nu(u_2)} \\
&\chi_{w_{u_2}, 1}(w_{u_1})^{\mathbf{1}_{u_1 \in \mathcal{I}_3', u_2 \in \mathcal{I}_3''} \cdot c(u_2) \nu(u_1)} \\
&\chi_{w_{u_1}, \log_p \ord(\chi(u_2))}(w_{u_2})^{\chi(u_2)(a(u_1))} \\
&\chi_{w_{u_2}, \log_p \ord(\chi(u_1))}(w_{u_1})^{\chi(u_1)(a(u_2))}.
\end{align*}
In the notation of Corollary \ref{cLargeSieve}, we therefore take
\begin{align*}
c_1 &:= \chi(u_2)(a(u_1)) + \mathbf{1}_{u_1, u_2 \in \mathcal{I}_3'} \cdot c(u_1) \nu(u_2) + \mathbf{1}_{u_1 \in \mathcal{I}_3'', u_2 \in \mathcal{I}_3'} \cdot c(u_1) \nu(u_2) \\
c_2 &:= \chi(u_1)(a(u_2)) + \mathbf{1}_{u_1, u_2 \in \mathcal{I}_3'} \cdot c(u_2) \nu(u_1) + \mathbf{1}_{u_1 \in \mathcal{I}_3', u_2 \in \mathcal{I}_3''} \cdot c(u_2) \nu(u_1).
\end{align*}
Since $u_1$ and $u_2$ are linked, we have $c_2 \neq 0$. This proves \eqref{eLSclaim}.

Second suppose that $U$ is unlinked, but $U$ is not equal to the set $\mathcal{U}$ from Theorem \ref{tUnlinkedClassification}. In this case, we apply the triangle inequality to deduce that
$$
N(X; U) \leq p \sum_{\substack{\mathbf{w} = (w_i)_{i \in \mathcal{I}_3}, \ \mathbf{f} = (f_a) \\ w_i > \exp\left((\log X)^{\frac{1}{10 p^3 |A|^5}}\right) \Longleftrightarrow i \in U}}^{\flat \flat \flat} \prod_{i \in \mathcal{I}_3} \left(\frac{1}{|A|}\right)^{\omega(w_i)} \prod_{i \in \mathcal{I}_3' \cup \mathcal{I}_3''} \left(\frac{1}{p}\right)^{\omega(w_i)}.
$$
Set
$$
t(i) := 
\begin{cases}
1 &\text{if } i \in \mathcal{I}_3 - (\mathcal{I}_3' \cup \mathcal{I}_3'') \\
p &\text{if } i \in \mathcal{I}_3' \cup \mathcal{I}_3''.
\end{cases}
$$
We now define
$$
\mathrm{wt}(U) := \sum_{i \in U} \frac{1}{t(i) |A| \varphi(\ord(a(i)))} + \frac{1}{10p^3 |A|^5} \sum_{i \in \mathcal{I}_3 - U} \frac{1}{t(i) |A| \varphi(\ord(a(i)))}.
$$
It follows from Shiu's theorem \cite{Shiu} that there is a constant $C > 0$, depending only on $A$, such that for all $X \geq C$
$$
N(X; U) \leq CX (\log X)^{\mathrm{wt}(U) - 1}.
$$
We now claim that 
\begin{equation}
\label{eWtClaim}
\mathrm{wt}(U) -1 < -1 + \sum_{a \in A - \{\text{id}\}} \frac{1}{\varphi(\ord(a))},
\end{equation}
where we remark that the RHS is the logarithmic exponent in M\"aki's theorem. Hence our claim implies that the contribution of these $U$ is negligible. Let us now prove the claim. Since $|\mathcal{I}_3| \leq p^2 |A|^3$, we have the bound
$$
\mathrm{wt}(U) \leq \frac{1}{10p |A|^2} + \sum_{i \in U} \frac{1}{t(i) |A| \varphi(\ord(a(i)))}.
$$
Note that the RHS now only increases upon increasing $U$. Moreover, if $U$ is a strict subset of $\mathcal{U}$, then we see that
$$
\frac{1}{10p |A|^2} + \sum_{i \in U} \frac{1}{t(i) |A| \varphi(\ord(a(i)))} < \sum_{i \in \mathcal{U}} \frac{1}{t(i) |A| \varphi(\ord(a(i)))} = \sum_{a \in A - \{\text{id}\}} \frac{1}{\varphi(\ord(a))}.
$$
Hence it suffices to prove that
$$
\frac{1}{10p |A|^2} + \sum_{i \in U} \frac{1}{t(i) |A| \varphi(\ord(a(i)))} < \sum_{a \in A - \{\text{id}\}} \frac{1}{\varphi(\ord(a))}
$$
for all maximal unlinked $U \neq \mathcal{U}$. But this follows from Theorem \ref{tUnlinkedClassification} and a direct computation, thus proving the claim \eqref{eWtClaim}.

Third, suppose that $U = \mathcal{U}$. We are now going to show that
\begin{equation}
\label{eFinalClaim}
N(X; \mathcal{U}) \sim \# \mathcal{F}(X),
\end{equation}
which clearly implies the theorem, as we have already shown that $N(X; U) = o(\# \mathcal{F}(X))$ for all $U \neq \mathcal{U}$.

We now perform one final dissection, this time of the sum $N(X; \mathcal{U})$. We let $C > 0$ be a large real number (in terms of $A$ only), and we define, for each partition $V_1 \cup V_2 \cup V_3 = \mathcal{I}_3 - \mathcal{U}$, the subsum $N(X; \mathcal{U}, V_1, V_2, V_3)$ of $N(X; \mathcal{U})$, where
$$
w_i = 1 \Longleftrightarrow i \in V_1, \quad 1 < w_i < (\log X)^C \Longleftrightarrow i \in V_2, \quad w_i \geq (\log X)^C \Longleftrightarrow i \in V_3.
$$
If $V_3$ is nonempty, then the large sieve, as codified in Corollary \ref{cLargeSieve}, yields
$$
N(X; \mathcal{U}, V_1, V_2, V_3) = o(\# \mathcal{F}(X))
$$
provided that we pick $C$ sufficiently large (in terms of $A$ only). Indeed, note that $\mathcal{U}$ is maximal unlinked, so if $V_3 \neq \varnothing$, then there exist $u \in \mathcal{U}$ and $v \in V_3$ that are linked (or alternatively, $v \in V_3$ is linked with $u \in \mathcal{U}$). Then the argument proceeds exactly as the one leading to \eqref{eLSclaim} except that $z$ is now an arbitrarily large logarithmic power.

Next, if $V_2$ is nonempty, then Siegel--Walfisz, see Lemma \ref{lSW}, gives
$$
N(X; \mathcal{U}, V_1, V_2, V_3) = o(\# \mathcal{F}(X)).
$$
Finally, we observe that
$$
N(X; \mathcal{U}, \mathcal{I}_3 - \mathcal{U}, \varnothing, \varnothing) = \sum_{\substack{\mathbf{w} = (w_i)_{i \in \mathcal{I}_3} \\ \mathbf{f} = (f_a), \quad w_i = 1 \, \forall i \in \mathcal{I}_3 - \mathcal{U} \\ w_i > \exp\left((\log X)^{\frac{1}{10 p^3 |A|^5}}\right) \, \forall i \in \mathcal{U}}}^{\flat \flat \flat} \sum_{\alpha \in \mathbb{F}_p} \tilde{h}(\mathbf{w}, \alpha).
$$
We unwrap the definition of $\tilde{h}$ as in equation \eqref{eDefTildeh}, and then we apply Lemma \ref{lSW} to see that the terms with $\alpha \neq 0$ are negligible. Hence we obtain that
$$
N(X; \mathcal{U}, \mathcal{I}_3 - \mathcal{U}, \varnothing, \varnothing) = \sum_{\substack{\mathbf{w} = (w_i)_{i \in \mathcal{I}_3} \\ \mathbf{f} = (f_a), \quad w_i = 1 \, \forall i \in \mathcal{I}_3 - \mathcal{U} \\ w_i > \exp\left((\log X)^{\frac{1}{10 p^3 |A|^5}}\right) \, \forall i \in \mathcal{U}}}^{\flat \flat \flat} \hspace{-0.26cm} \prod_{i \in \mathcal{I}_3} \left(\frac{1}{|A|}\right)^{\omega(w_i)} \hspace{-0.26cm} \prod_{i \in \mathcal{I}_3' \cup \mathcal{I}_3''} \left(\frac{1}{p}\right)^{\omega(w_i)} + o(\# \mathcal{F}(X)).
$$
We now drop the condition 
$$
w_i > \exp\left((\log X)^{\frac{1}{10 p^3 |A|^5}}\right)
$$
incurring a negligible error (and similarly for the height bound involving $G(X)$ in the summation conditions stemming from $\flat \flat \flat$). The remaining sum can be rewritten as $\sum_{n \leq X} f(n)$ with $f$ the multiplicative function supported on squarefrees and given on the primes by
$$
f(q) =
\begin{cases}
|A| - 1 &\text{if } q = p \\
|\{a \in A - \{\text{id}\} : q \equiv 1 \bmod \ord(a)\}| &\text{if } q \neq p.
\end{cases}
$$
Evaluating this sum with \cite[Theorem 13.2]{Kou} gives equation \eqref{eFinalClaim} and hence the theorem.
\end{proof}

\subsection{Proof of Theorem \ref{thm:main}}
We will now restate the main result of our paper.

\begin{theorem}\label{thm:nr-Hom}
Let $p$ be an odd prime and let $A$ be a finite abelian $p$-group. Let $e$ be a nontrivial primitive idempotent of $\Q_p[A]$. Let $1 \leq j \leq r_e$ and let $M_j = e \Z_p[A]/\mathfrak{m}_e^j$. Then we have for 100\% of the $\varphi \in \mathrm{Epi}(G_\Q, A)$, ordered by their product of ramified primes $\mathfrak{f}(\varphi)$, the equality
$$
|\Hom_{\textup{nr}}(G_K, M_j)^A| = \frac{|N_{\textup{typical}}(A, M_j)|}{|M_j| \cdot |H^1(A, M_j)|} \frac{|H^0(G_\Q, M_j)|}{|H^0(G_\Q, M_j^\ast)|} \prod_{v \mid \mathfrak{f}(\varphi)} p^{\max\{1 \leq i \leq j : v \textup{ is special at level } i\}},
$$
where $K$ denotes the fixed field of $\ker(\varphi)$.
\end{theorem}

\begin{proof}
For each $1 \leq k \leq j$, let $\mathcal{L}_{v, k}$ be the local condition defined in \eqref{eq:L-v-j}. Write $f_k: M_k \rightarrow \mathbb{F}_p$ for the natural quotient map. We define the new local conditions on $\mu_p$ given by
$$
\mathcal{L}_{v, k}' = 
\begin{cases}
{f_k^*}^{-1}(\mathcal{L}_{v, k}^\perp) &\text{if } v \neq p \\
H^1(G_{\Q_p}, \mu_p) &\text{if } v = p,
\end{cases}
$$
and we write $\mathcal{L}_k'$ for the resulting Selmer structure.

We now apply Theorem \ref{tMainAnalytic} to each $\mathcal{L}_k'$ for $1 \leq k \leq j$. We must check the hypotheses of Theorem \ref{tMainAnalytic} from Definition \ref{dLocal}: 
\begin{itemize}
\item the first bullet point is a direct computation with the local conditions (note that $\mathcal{T}_v$ acts trivially on $M_k^\ast$ for $v \nmid p \mathfrak{f}(\varphi)$ and that $f_k^\ast$ is injective, hence we have the equality $(f_k^\ast)^{-1}(H^1_{\text{nr}}(G_{\Q_v}, M_k^\ast)) = H^1_{\text{nr}}(G_{\Q_v}, \mu_p)$),
\item the second bullet point is Lemma \ref{lLocal},
\item the third bullet point holds by Proposition \ref{prop:pushforward<p},
\item the fourth bullet point holds by definition of $r_e$. Indeed, by definition of $r_e$, there exists $\gamma \in A$ such that $1 - \gamma$ and $\ord(\gamma)$ annihilate $M_{r_e}$, and hence $M_k$ (as $k \leq j \leq r_e$). Then we can take $C = \langle \gamma \rangle$, which is a nontrivial subgroup,
\item the fifth bullet point is true by our choice of $\mathcal{L}_{p, k}'$.
\end{itemize}
Hence Theorem \ref{tMainAnalytic} shows that for 100\% of the $\varphi \in \mathrm{Epi}(G_\Q, A)$, ordered by their product of ramified primes $\mathfrak{f}(\varphi)$, we have
$$
\mathrm{Sel}_{\mathcal{L}_k'}(G_\Q, \mu_p) = 0
$$
for all $1 \leq k \leq j$. For such $\varphi$, we get vanishing of the dual Selmer group 
\begin{equation}
\label{eDualVanish}
\mathrm{Sel}_{\mathcal{L}_j^\perp}(G_\Q, M_j^\ast) = 0
\end{equation}
by Lemma \ref{lem:Selmer-vanish}. Hence the hypothesis of Lemma \ref{lStatisticalFormula} is satisfied. This lemma yields
$$
|\Hom_{\textup{nr}}(G_K, M_j)^A| = \frac{|N_\varphi(A, M_j)| |\mathrm{Sel}_{\mathcal{L}_j}(G_\Q, M_j)|}{|H^1(A, M_j)|}.
$$
We apply Lemma \ref{lFormulaN} to replace $|N_\varphi(A,M_j)|$ by $|N_{\text{typical}}(A, M_j)|$. Finally, we apply the Greenberg--Wiles formula to understand $\mathrm{Sel}_{\mathcal{L}_j}(G_\Q, M_j)$, which gives
\begin{align*}
|\mathrm{Sel}_{\mathcal{L}_j}(G_\Q, M_j)| &= \frac{|\mathrm{Sel}_{\mathcal{L}_j^\perp}(G_\Q, M_j^\ast)|}{|M_j|} \frac{|H^0(G_\Q, M_j)|}{|H^0(G_\Q, M_j^\ast)|} \prod_{v \mid \mathfrak{f}(\varphi)} \frac{|\mathcal{L}_{v, j}|}{|H^0(G_{\Q_v}, M_j)|} \\
&= \frac{|H^0(G_\Q, M_j)|}{|M_j| \cdot |H^0(G_\Q, M_j^\ast)|} \prod_{v \mid \mathfrak{f}(\varphi)} p^{\max\{1 \leq i \leq j : v \textup{ is special at level } i\}}
\end{align*}
by equation \eqref{eDualVanish} and Proposition \ref{prop:local-cond-size} to relate the local factors directly to special primes. Here the additional factor $|M_j|$ comes from the infinite place in the Greenberg--Wiles formula. Indeed, the infinite place contributes
$$
\frac{|\mathcal{L}_{\infty, j}|}{|H^0(G_{\mathbb{R}}, M_j)|} = \frac{1}{|M_j|},
$$
because $p$ is odd.
\end{proof}

We can also give an explicit expression for the constant in Theorem~\ref{thm:main}.

	\begin{theorem}\label{thm:main-complete}
		Theorem~\ref{thm:main} holds with 
		\[
			C(A, e\Z_p[A]/I)= \begin{cases}
				 \log_p \dfrac{|N_{\textup{typical}}(A, e\Z_p[A]/I)|}{|N_{\textup{typical}}(A, \mathfrak{m}_ee\Z_p[A]/I)|}-1 & \textup{ if $e\Z_p[A]/I \not \cong \FF_p$} \\
				 - \dim_{\FF_p} A[p] & \textup{ otherwise.}
				 \end{cases}
		\]
	\end{theorem}
	
	\begin{proof}
		Note that $\Hom_{\nr}(G_K, M_j)^A= \Hom_{\Z_p[A]}(\CL(K), M_j)$. Since $M_j$ is an $e\Z_p[A]$-module and $e\CL(K)=\CL(K)\otimes_{\Z_p[A]}e\Z_p[A]$, we have 
		\[
			\Hom_{\nr}(G_K, M_j)^A= \Hom_{e\Z_p[A]}(e\CL(K), M_j).
		\]
		Write $e\CL(K)\cong \bigoplus_{i=1}^{\infty} (e\Z_p[A]/\mathfrak{m}_e^i)^{\oplus n_i}$. Then
		\begin{equation}\label{eq:Hom-nr}
			\Hom_{\nr}(G_K, M_j)^A= \Hom_{e\Z_p[A]}(e\CL(K), M_j) =  \bigoplus_{i=1}^{j-1} (e\Z_p[A]/\mathfrak{m}_e^i)^{\oplus n_i} \oplus \bigoplus_{i=j}^{\infty} (e\Z_p[A]/\mathfrak{m}_e^i)^{\oplus n_i}.
		\end{equation}
		Because $e\Z_p[A]$ is a discrete valuation ring with residue field $\FF_p$, $|e\Z_p[A]/\mathfrak{m}_e^i|=p^i$.
		Suppose $I=\mathfrak{m}_e^d$. Since $I$ is assumed to be proper, we have $d\geq 1$. By \eqref{eq:Hom-nr} and the definition of $I$-rank, we have
		\begin{equation}\label{eq:rk-Hom}
			\rk_I e\CL(K) = \log_p \frac{|\Hom_{\nr}(G_K, M_d)^A|}{|\Hom_{\nr}(G_K, M_{d-1})^A|},
		\end{equation}
		and we will use the formula in Theorem~\ref{thm:nr-Hom} to compute both $|\Hom_{\nr}(G_K, M_{d - 1})^A|$ and $|\Hom_{\nr}(G_K, M_d)^A|$ for 100\% of the epimorphisms $\varphi \in \mathrm{Epi}(G_\Q, A)$.	
		
		By \cite[Lemma~2.5]{Liu}, the $A$-action on $e\Z_p[A]$ factors through a nontrivial cyclic quotient of $A$. So $A$ can be decomposed as $H \times C$ for some cyclic subgroup $C$ such that $H$ acts trivially on $e\Z_p[A]$ (this decomposition is not unique). Recall that $M_j^C \cong \FF_p$. Therefore we have by the Hochschild--Serre spectral sequence \cite[Theorem~2.4.6]{NSW}
		\[
			H^1(A, M_j) \cong H^1(H, M_j^C) \oplus H^1(C, M_j) = H^1(H, \FF_p) \oplus H^1(C, M_j)
		\]
		for all $j \geq 1$.
		Since $C$ is cyclic, it follows by \cite[Proposition~1.7.6]{NSW} that $|H^1(C, M_j)| = |\widehat{H}^0(C, M_j)| = |M_j^C/\textup{Nm}_C M_j|$. Using \cite[Lemma~2.6]{Liu} and noting $C$ acts nontrivially on $e\Z_p[A]$, the norm map $\textup{Nm}_C$ annihilates $e\Z_p[A]$. Since $M_j^C \cong \FF_p$, we obtain 
		\[
			|H^1(A, M_j)| = |H^1(H, \FF_p)| \cdot | M_j^C/\textup{Nm}_C M_j| = p \cdot |H^1(H, \FF_p)|,
		\]
		which does not depend on $j$ as long as $j \geq 1$.
		
		Using the above results, we apply Theorem~\ref{thm:nr-Hom} to compute \eqref{eq:rk-Hom}. First, assume $d \geq 2$ (i.e., $I \subset \mathfrak{m}_e$). Since $e$ is a nontrivial idempotent, $M_d^A= M_{d-1}^A=\FF_p$, and therefore $H^0(G_{\Q}, M_d) \cong H^0(G_{\Q}, M_{d-1}) = \FF_p$. Note that $e\Z_p[A]/\mathfrak{m}_e \cong \FF_p$ is the unique simple quotient of $e\Z_p[A]$, so $\mu_p$ is the unique minimal submodule of $M_d^*$ and $M_{d-1}^*$. Then since $p$ is odd, $\mu_p \not \subseteq \Q$ and we have $H^0(G_{\Q}, M_d^*)=H^0(G_{\Q}, M_{d-1}^*)=0$. By Theorem~\ref{thm:nr-Hom}, we obtain
		\[
			\rk_I e\CL(K) = \log_p \frac{|N_{\textup{typical}}(A, M_d)|}{|N_{\textup{typical}}(A, M_{d-1})|}-1 + \#\left\{v \mid \mathfrak{f}(\varphi) : \text{$v$ is special at level $d$}\right\}.
		\]

		When $d=1$ (i.e., $I=\mathfrak{m}_e$), $e\Z_p[A]/I=M_d=\FF_p$ and $M_{d-1}=0$, so for 100\% of the epimorphisms $\varphi \in \mathrm{Epi}(G_\Q, A)$
		\begin{eqnarray*}
			\rk_I e\CL(K) &=& \log_p|\Hom_{\nr}(G_K, \FF_p)^A| \\
			&=& \log_p \frac{|N_{\textup{typical}}(A, \FF_p)|}{|H^1(A, \FF_p)|} + \#\left\{v \mid \mathfrak{f}(\varphi) : \text{$v$ is special at level $1$}\right\}\\
			&=& \#\left\{v \mid \mathfrak{f}(\varphi) : \text{$v$ is special at level $1$}\right\} - \dim_{\FF_p} A[p],
		\end{eqnarray*}
		where the last step uses Lemma~\ref{lem:N-Fp}. 
		Finally, the theorem follows immediately from Definition~\ref{def:special-prime}.
	\end{proof}
	
	We now give the proof of Corollary~\ref{cor:multicyclic} and Corollary~\ref{cor:cyclic}.
	
	\begin{proof}[Proof of Corollary~\ref{cor:multicyclic}]
		Let $A=(\Z/p\Z)^{\oplus n}$ with $n >1$, and $e$ a nontrivial primitive idempotent of $\Q_p[A]$. Recall that the action of $A$ factors through a nontrivial cyclic quotient; this quotient has to be isomorphic to $\Z/p\Z$. Let $H$ denote the kernel of the quotient map from $A$ to this cyclic quotient. So $H \cong (\Z/p\Z)^{\oplus n-1}$. One can check by definition~\eqref{eq:def-Ie} that $I_e=p\cdot e\Z_p[A]= \mathfrak{m}_e^{p-1} e\Z_p[A]$. When $I=\mathfrak{m}_e$, the statement $\ref{item:multicyclic-1}$ of the corollary follows from genus theory. When $I \subset \mathfrak{m}_e$, $A$ acts nontrivially on $e\Z_p[A]/I$, and $A/H$ acts faithfully on $e\Z_p[A]/I$. So the special primes are exactly those ramified primes with $\iota(D_v) \subseteq H$. Then the statement $\ref{item:multicyclic-2}$ follows from Theorem~\ref{thm:main-complete} and Lemma~\ref{lem:N-M-exp=p}.
	\end{proof}
	
	\begin{proof}[Proof of Corollary~\ref{cor:cyclic}]
		Since $A$ is cyclic, Lemma~\ref{lem:N-in-inf} implies that $N_{\textup{typical}}(A, e\Z_p[A]/I) =0$ for any nonzero ideal $I$.
		Then the formula for $\rk_I e\CL(K)$ in the corollary directly follows from Theorem~\ref{thm:main-complete}. 
	\end{proof}

\end{document}